\documentclass[11pt]{article}

\usepackage{amsthm,amsmath,amssymb}
\usepackage[normalem]{ulem}
\usepackage[bookmarks=true,hypertexnames=false,pagebackref]{hyperref}
\hypersetup{colorlinks=true, citecolor=blue, linkcolor=red,
  urlcolor=blue}

\usepackage{newpxtext}
\usepackage[libertine,vvarbb]{newtxmath}

\usepackage{cancel}

\usepackage[T1]{fontenc} 
\usepackage{textcomp} 

\usepackage{thm-restate}

\usepackage[scr=rsfso]{mathalfa}
\usepackage{nicefrac}
\usepackage{cleveref}
\usepackage{graphicx}
\usepackage[ruled,vlined]{algorithm2e}
\usepackage{aliascnt}
\usepackage[section]{placeins}
\usepackage{tcolorbox}
\usepackage{epigraph}
\usepackage{algorithmic}
\usepackage{float}
\usepackage{xifthen}
\usepackage{latexsym}
\usepackage{framed}
\usepackage{fullpage}
\usepackage{bm}
\usepackage[normalem]{ulem}

\newtheorem{theorem}{Theorem}[section]

\newtheorem{fact}[theorem]{Fact}
\newtheorem{corollary}[theorem]{Corollary}
\newtheorem{lemma}[theorem]{Lemma}
\newtheorem{definition}[theorem]{Definition}

\newtheorem{remark}[theorem]{Remark}

\newtheorem{expr}{Experiment}
\newtheorem{Algo}{Algorithm}

\SetKwInOut{Oracle}{Oracle}
\SetKwInOut{Parameter}{parameter}
\ResetInOut{output}

\newcommand{\sdotfill}{\textcolor[rgb]{0.8,0.8,0.8}{\dotfill}} 

\newcommand{\bu}{\boldsymbol{u}}

\newcommand{\bv}{\boldsymbol{v}}

\newcommand{\meanA}{\overline{A}}
\newcommand{\meanB}{\overline{B}}
\newcommand{\eps}{\varepsilon}

\newcommand{\OO}{\mathcal{O}}
\newcommand{\E}{\ensuremath{\mathrm{E}}}
\newcommand{\SBM}{\ensuremath{\mathrm{SBM}}}
\newcommand{\SSBM}{\ensuremath{\mathrm{SSBM}}}
\newcommand{\USSBM}{\ensuremath{\mathrm{USSBM}}}

\definecolor{corall}{RGB}{200,100,100}
\definecolor{teall}{RGB}{0,212,212}

\newcommand{\smax}{s^{*}}
\newcommand{\smin}{s_{\min}}
\newcommand{\W}{W^t_{\ell_0,\ell_{t+1}}}
\newcommand{\Ell}{\mathcal{L}}

\newcommand{\Z}{Z^t_{a,b}}

\newcommand{\mP}[1]{\mathcal{P}^{(#1)}}
\newcommand{\mmP}{\mathcal{P}}

\newcommand{\mX}{\mathcal{X}}

\newcommand{\abs}[1]{\left|#1\right|}

\usepackage{eqparbox}

\definecolor{mypink}{RGB}{219, 48, 122}

\title{Detecting Hidden Communities by Power Iterations with Connections to Vanilla Spectral Algorithms}

\author{
Chandra Sekhar Mukherjee 
\thanks{Research supported by NSF CAREER award 2141536.}
\\
Department of Computer Science\\
University of Southern California\\
\texttt{chandrasekhar.mukherjee@usc.edu}
\and
Jiapeng Zhang
\thanks{Research supported by NSF CAREER award 2141536.}
\\
Department of Computer Science\\
University of Southern California\\
\texttt{jiapengz@usc.edu}
}

\date{\today}

\begin{document}
\maketitle

\begin{abstract}
Community detection in the stochastic block model is one of the central problems of graph clustering. Since its introduction by Holland, Laskey and Leinhardt (Social Networks, 1983), many subsequent papers have made great strides in solving and understanding this model. In this setup, spectral algorithms have been one of the most widely used frameworks for the design of clustering algorithms. However, despite the long history of study, there are still unsolved challenges.  
One of the main open problems is the design and analysis of ``simple''(vanilla) spectral algorithms, especially when the number of communities is large. 

In this paper, we provide two algorithms. The first one is based on the power-iteration method. This is a simple algorithm which only compares the rows of the powered adjacency matrix. Our algorithm performs optimally (up to logarithmic factors) compared to the best known bounds in the dense graph regime by Van Vu (Combinatorics Probability and Computing, 2018). Furthermore, our algorithm is also robust to the ``small cluster barrier'', recovering large clusters in the presence of an arbitrary number of small clusters. Then based on a connection between the powered adjacency matrix and eigenvectors, we provide a ``vanilla'' spectral algorithm for large number of communities in the balanced case. This answers an open question by Van Vu (Combinatorics Probability and Computing, 2018) in the balanced case. Our methods also partially solve technical barriers discussed by Abbe, Fan, Wang and Zhong (Annals of Statistics, 2020).

In the technical side, we introduce a random partition method to analyze each entry of a powered random matrix. This method can be viewed as an eigenvector version of Wigner's trace method. Recall that Wigner's trace method links the trace of powered matrix to eigenvalues. Our method links the whole powered matrix to the span of eigenvectors. We expect our method to have more applications in random matrix theory.

\end{abstract}

\section{Introduction}
Community detection (graph clustering) is a fundamental problem in computer science that has applications in diverse areas including but not limited to physics, social sciences and biology. Among others, the \emph{stochastic block model} (SBM) is one of the most widely used frameworks for studying community detection, offering both theoretical arena for rigorous analysis of performance of clustering algorithms as well serving as a simulation benchmark in practice.
Before proceeding ahead we briefly describe the model. There is a set of $n$ vertices $V=\{v_1, \ldots ,v_n\}$ with a hidden partition $\{V_1, \ldots ,V_k\}$. A graph $G$ is then sampled through a random process, governed by a $k \times k$ symmetric matrix $W$. For any pair of vertices $v \in V_i$ and $u \in V_j$ an edge is added between them independently at random with probability $W_{i,j}$. This sampling process is called $SBM(n,\{V_i\}_{i=1}^k,W)$
Naturally, intra cluster edge probabilities are higher than inter cluster edge probabilities. Given such a sampled graph $G(V,E)$ (its adjacency matrix being denoted as $A$), the task is to recover the hidden partition. 

From the perspective of theoretical computer science,
there has been a lot of progress towards understanding the information theoretic and computational limits of solving the SBM problem since the 1980s~\cite{holland1983stochastic,bui1987graph,dyer1989solution,boppana1987eigenvalues}. We refer to the recent survey by Abbe \cite{abbe2017community} for a comprehensive list of such results.
In this paper we study one of the prominent problems in relation to SBM, 
which is designing of simple spectral algorithm for community detection.
Spectral algorithms have a long history of being used in the random graph paradigm~\cite{boppana1987eigenvalues, alon1998finding, mcsherry2001spectral, vu2018simple, abbe2020entrywise, mukherjee2022recovering}. 
Several papers in this area have concentrated on finding simplest spectral algorithms, (sometimes called vanilla algorithms \cite{abbe2020entrywise}) that can solve the SBM problem~\cite{mcsherry2001spectral,vu2018simple,abbe2020entrywise}.

\paragraph{Vanilla spectral algorithms: motivation and challenges}
Spectral algorithms broadly refer to algorithms that analyze the eigenvectors of the (random) adjacency matrix $A$. 
In this direction most works either analyze the entries of the eigenvectors or the projection of $A$ onto the eigenspace of the top eigenvectors. In general, there are several steps that precede and/or follow the obtaining of the eigenvectors/eigenspace-projection. For example, the projection is sometimes preceded by trimming steps (such as random partition in \cite{vu2018simple}), and the post projection embedding is then subjected to further iterative cleaning  (\cite{Yun2014AccurateCD}) , or other  clustering techniques such as K-means  (\cite{JingconsistentSBM}) are then implemented. 
The goal of obtaining a ``vanilla'' spectral algorithm is to avoid such steps. On a high level, there are two motivations. 
The first is to understand if an algorithm that just applies some threshold on the entries of the eigenvectors or the rows of the projected matrix is able to recover the partition. This also gives intuition about power of spectral algorithms used in practice that are often ``vanilla'' in this sense (such as PCA). 

The next is technical motivation. As we shall explain going forward, the difficulty in obtaining a provably correct vanilla algorithm is due to certain technical barriers in random matrix theory. Thus, analysis of vanilla spectral algorithm promises to improve our understanding of these hurdles which may have application in other problems/domains. One may refer to Section~\ref{sec: contribution-2} for a more detailed description on motivations and progress made so far.

Although there have been several efforts and progresses made towards designing such algorithms~\cite{abbe2020entrywise,Eldrige-Unperturbed,vu2018simple}, the problem of obtaining a vanilla spectral algorithm for large $k$ remains open. 
We describe the approaches and limitations of existing work in Section~\ref{sec: connections to others}.

Motivated by these questions and hurdles, in this paper we continue on the direction of obtaining vanilla spectral algorithms for SBM through resolving barriers in random matrix theory.  In our first attempt, we focus on the symmetric SBM framework ($\SSBM$).

\begin{definition}[Symmetric SBM]
\label{def:ssbm}
In this model, given an $n$-vertex set $V$
with a hidden partition $V=\cup_{i=1}^k V_i$ such that $V_i \cap V_j= \emptyset $ for all $i \ne j$, we say a graph $G=(V,E)$ is sampled from $\SSBM(n,k,p,q)$, if for all pairs of vertices $v_i,v_j \in V$, 
\begin{itemize}
    \item an edge $(v_i,v_j)$ is added independently with probability $p$, if $v_i,v_j \in V_{\ell}$ for some $\ell$;
    \item an edge $(v_i,v_j)$ is added independently with probability $q$,  otherwise. 
\end{itemize}
\end{definition}

Although $\SSBM$ is a basic version of $\SBM$, due to its simplicity it's often used as a starting point for design of community detection algorithms, as in~\cite{abbe2020entrywise}. 
Furthermore, we work in the ``dense graph'' setting (
$q\gg \frac{\log n}{n}$) 
which is commonly the area of interest for spectral algorithms~\cite{mcsherry2001spectral,vu2018simple,abbe2020entrywise}.

Against this backdrop, we now describe the results of our paper on a high level.

\subsection{Our Contribution}

Our contribution is two part. 

In the algorithm part, we design a simple algorithm for the $\SSBM$ model based on powering of the adjacency matrix of $G$.  Our algorithm is simple,  and matches the best known algorithms for large $k$~\cite{vu2018simple,mukherjee2022recovering} upto logarithmic factor even in the presence of an arbitrary number of small (size $<< \sqrt{n}$) clusters.  Next, by connecting power matrix and eigenvectors of symmetric matrix, we convert our power iteration algorithm into a ``vanilla'' spectral algorithm for the balanced $\SSBM$ case, making significant progress w.r.t open questions (in the balanced case) raised by \cite{vu2018simple} and \cite{abbe2020entrywise}.
More details will be discussed in Section \ref{sec: contribution-2}.

In doing so, we also make contribution in technical analysis.  We give an analysis of eigenspaces of random matrices by using power matrix, providing another perspective to understand random matrices.  Our analysis is built on a random partition idea.  
We discuss more connections between our analysis and known analysis in Section \ref{sec: connections to others}.

\subsubsection{Our Algorithm to Recover the Largest Partition}
\label{sec:contribution-1}
Our first contribution is a simple algorithm based on the power iteration method. Given the adjacency matrix $A$ of a graph $G$ drawn in from $\SSBM(n,k,p,q)$, we obtain $B$ by subtracting each entry of $A$ by $q$.
That is, $B:=A-q\cdot 1^{n \times n}$, where $1^{n \times n}$ is all one $n \times n$ matrix . We show that the largest partition of $G$ (under reasonable constraints) can be recovered by comparing the euclidean distance of rows of $B^r$ for $r \approx \log n$. We describe the procedure in Algorithm~\ref{alg:powermethod}.

For a vertex $v_i\in V$, let the $i$-th row of $A$ (and $B$) correspond to $v_i$, and $V_1$ be the largest (hidden) partition of $G$. For any fixed $v_i \in V_1$, we show that there is a threshold $\Delta$ such that with high probability for all $v_j \in V_1$, the distance between the $i$-th and $j$-th row of $B^r$ is less than $\Delta$. Otherwise ($v_j\not\in V_1$) the distance is larger than $1.1\Delta$, which is sufficient to recover $V_1$. 
\begin{algorithm}
\caption{\textsc{Detecting Communities by Power Iterations}}
\label{alg:powermethod}
\begin{algorithmic}[1]
\STATE \textbf{Input:} A graph $G$ and parameters $p,q>0$.
\STATE Let $A$ be the adjacent matrix of $G$
\STATE $B\gets A - q\cdot \mathbf{1}^{n\times n}$ \COMMENT{ (Here $\mathbf{1}^{n\times n}$ is the all $1$ matrix.)}
\STATE Let $\Delta>0$ and $r>1$ be parameters \COMMENT{ (We will explain how to choose them later)}
\FOR{ $v_i,v_j \in G$}
\IF{$\|B^{r}_i - B^{r}_j\|_{2}\leq \Delta$}
\STATE Put $v_i$ and $v_j$ in a same cluster \COMMENT{($B^r_i$ represents the $i$-th row of $B^r$)}
\ENDIF
\ENDFOR
\STATE Output the sets thus formed.
\end{algorithmic}
\end{algorithm}

The correctness of Algorithm~\ref{alg:powermethod} can be formalized as follows. 
\begin{theorem}[Recovering largest cluster]
\label{thm: main}
There are constants $C, C_0>0$ such that the following holds. 
Let $ p,q \le 0.75$ be parameters 
such that $\max \{p(1-p),q(1-q) \} \ge C_0(\log n)/n.$  Let $G$ be a random graph sampled from $\SSBM(n,k,p,q)$, and let $\smax$ be the size of the largest cluster. If 
$\smax \ge C \cdot (\log n)^7 \cdot \sqrt{p(1-q)} \cdot \sqrt{n}/(p-q)$, then with probability $1-\OO(1/n)$ one of the cluster output by Algorithm~\ref{alg:powermethod} is the largest cluster of $G$ for suitable choice of $\Delta$ and $r$.
\end{theorem}
\paragraph{Comparison to best known result.}
In the line of spectral algorithms, the state of the art for large $k$ is provided by Vu \cite{vu2018simple} which proved that if each cluster has size larger than  $C \cdot (\log n) \cdot \sqrt{p(1-q)} \cdot \sqrt{n}/(p-q)$, then a spectral algorithm can recover all clusters. 
In fact, most community recovery algorithms require ``all'' clusters to be large. 
This constraint was termed as the ``small cluster barrier'' by Ailon et al.~\cite{ailon2013breaking}.
Recently the work by Mukherjee et al.~\cite{mukherjee2022recovering} provided an algorithm for this problem building on Vu's algorithm that does not need any size constraint on the clusters that are not the largest. They showed that if $(p-q)\smax \ge C_0\sqrt{p(1-q)} \sqrt{n} \log n$, then their algorithm is able to recover the largest cluster.
Our algorithm is also able to bypass the small cluster barrier and recover the largest cluster in the presence of small clusters. Our algorithm has an additional $(\log n)^6$ factor than \cite{mukherjee2022recovering} but is far simpler in comparison.

\subsubsection{Towards Vanilla Spectral Algorithms}
\label{sec: contribution-2}
Now we discuss the study of simple (vanilla) spectral algorithms, further describing the motivation, progress made so far and the open problems, before recording our contribution.

There are two primary motivations. 
The first, as stated in \cite{vu2018simple} is to design theoretically provable natural algorithms. Many spectral algorithms used in practice, such as PCA do not have any pre-processing or post-processing steps akin to that of
~\cite{JingconsistentSBM,vu2018simple,Yun2014AccurateCD}.
The second is to simply understand the power of just SVD projection in community recovery, and understanding the extent of its success w.r.t information theoretic bounds, which was the motivation in \cite{abbe2020entrywise}. 

Beyond these, since the present analysis is limited by certain hurdles in random matrix theory (as we discuss in later sections), we hope that studying this problem also furthers our understanding of random matrices, which is the technical motivation of this paper.

Against this background, Abbe et al. in their beautiful work~\cite{abbe2020entrywise} designed a vanilla spectral algorithm for $k=2$ clusters that recovers up to the information-theoretic bound. The algorithm is very simple. Obtain the second eigenvector of $A$, and if an entry is positive, then the corresponding vertex belongs to one cluster, otherwise to a different cluster. However, they could not extend analysis for $k \ge 3$.

\paragraph{Spectral algorithms for large $k$:}
In this paper we concentrate on the case of $k=\omega( \log n)$. In this domain the seminal work by McSherry~\cite{mcsherry2001spectral} provided a spectral algorithm to solve the hidden partition problem. However, this algorithm required a combinatorial and iterative partition of the vertices based on SVD projection of $A$. McSherry asked whether this step can be simplified, and this was answered by Vu~\cite{vu2018simple}.
To our best knowledge, this is the simplest spectral algorithm for $k=\omega( \log n)$. We describe this in Algorithm~\ref{alg:vualgorithm}, described as the SVD-II algorithm by Vu \footnote{This is not exactly the same as Vu's algorithm but captures the main idea.}.

\begin{algorithm}[ht]
\caption{\textsc{Vu's Spectral Algorithm to Detect Hidden Communities (SVD II in \cite{vu2018simple})}}
\label{alg:vualgorithm}
\begin{algorithmic}[1]
\STATE \textbf{Input:} A graph $G=(V,E)$ and parameters $p,q,k>0$.
\STATE Let $A$ be the adjacent matrix of $G$
\STATE Randomly partition $V$ into two parts $V_1$ and $V_2$
\STATE Let $A_1$ be the adjacency matrix of the subgraph induced by $V_1$
\FOR{ $\bu\in V_2 $}
\STATE Let $P^k_{A_1}\bu$ be the projection of $\bu$ into the space spanned by the first $k$ eigenvectors of $A_1$
\ENDFOR
\STATE Choose a parameter $\Delta>0$ 
\FOR{ $\bu,\bv\in V_2$}
\IF{$\|P^k_{A_1}\bu - P^k_{A_1}\bv\|_{2}\leq \Delta$}   
\STATE Put $\bu$ and $\bv$ in a same cluster
\ENDIF
\ENDFOR
\end{algorithmic}
\end{algorithm}

In brief, the adjacency matrix $A$ is partitioned randomly in two halves $A_1$ and $A_2$, and then columns of $A_2$ are projected onto the first $k$ eigenvectors  of $A_1$, and then the post projected columns of $A_2$ can be separated by clustering via distance based on some threshold. 
The key reason behind this partition is that the eigenvectors of $A_1$ and the columns of $A_2$ are uncorrelated, which aids the analysis. Once $A_2$ is clustered, recovering the communities from $A_1$ is relatively straightforward. Although this algorithm is simpler in comparison to \cite{mcsherry2001spectral}, there is a step that seems redundant. Specifically, the \emph{partition before projection} step. This was also noted by Vu~\cite{vu2018simple}, and he proposed the following algorithm, that he named SVD-I.

\paragraph{SVD-I and the barrier in analysis}

First we describe the conjectured algorithm by Vu~\cite{vu2018simple} in Algorithm~\ref{alg:vuconjecture}.

\begin{algorithm}
\caption{\textsc{SVD I in \cite{vu2018simple}}}
\label{alg:vuconjecture}
\begin{algorithmic}[1]
\STATE \textbf{Input:} A graph $G=(V,E)$ and parameters $p,q,k>0$.
\STATE Let $A$ be the adjacency matrix of $G$
\STATE Let $P^k_{A}$ be the projection operator into the space spanned by the first $k$ eigenvectors of $A$
\STATE Choose a parameter $\Delta>0$ 
\FOR{ $\bu,\bv\in V$}
\IF{$\|P^k_{A}\bu - P^k_{A}\bv\|_{2}\leq \Delta$} 
\STATE Put $\bu$ and $\bv$ in a same cluster
\ENDIF
\ENDFOR
\end{algorithmic}
\end{algorithm}

Note that compared to SVD-II, SVD-I does not have any partition. $A$ is simply projected onto the space of its first $k$ eigenvectors. The main barrier in analysis is as follows. In SVD-I, the eigenvectors and columns of $A$ are correlated. This was also pointed out by Vu.
\begin{quote}
\emph{Compared to SVD I, the extra steps in SVD II are the random partitions, done in order to reduce correlation.}
\end{quote}

In this regard, SVD-I is a more natural and intuitive algorithm, akin to widely popular algorithms used in practice such as principle component analysis (PCA).
The problem of whether SVD-I can be proven to be correct was left as an open problem in \cite{vu2018simple}. In this note Vu remarked
\begin{quote}
\emph{While SVD I could well win the contest for being the simplest algorithm, and perhaps the first one that most practitioners of the spectral method would think of, it is hard to analyse in the general case. In what follows, we analyse a slightly more technical alternative, SVD II}
\end{quote}

This presence of correlation is the technical barrier that we address in this paper.
We show that a slightly different algorithm from SVD-I works in the balanced case using our power iteration based analysis.

\paragraph{The $\USSBM$ model and our centered-SVD algorithm}

The $\USSBM$ setup is the balanced case of the symmetric SBM case. That is, if there are $k$ hidden partitions, then each vertex is assigned a partition uniformly at random. This implies the size of the partitions are fairly balanced. We note that the $\USSBM$ could be the first step towards the general $\SBM$ problem. 

In this paper, our main focus is on addressing the correlation barrier in \cite{vu2018simple}. In this model we show Algorithm~\ref{alg:oursvd} recovers the hidden communities under conditions comparable to that of the more complex SVD-II algorithm by Vu~\cite{vu2018simple}. 

\begin{algorithm}
\caption{\textsc{Centered-SVD}}
\label{alg:oursvd}
\begin{algorithmic}[1]
\STATE \textbf{Input:} A graph $G=(V,E)$ and parameters $p,q,k>0$.
\STATE Let $A$ be the adjacent matrix of $G$
\STATE $B\gets A - q\cdot \mathbf{1}^{n\times n}$
\COMMENT{(This is the extra step compared to SVD-I)}
\STATE Let $P^k_{B}$ be the projection operator into the space spanned by the first $k$ eigenvectors of $B$
\STATE Choose a parameter $\Delta>0$ depending on $n,k,p,q$
\FOR{ $\bu,\bv\in V$}
\IF{$\|P^k_{B}\bu - P^k_{B}\bv\|_{2}\leq \Delta$}   
\STATE Put $\bu$ and $\bv$ in a same cluster
\ENDIF
\ENDFOR
\end{algorithmic}
\end{algorithm}

As we aim for a vanilla algorithm, the idea is very simple. Given the adjacency matrix $A$, shift each entry by $q$ to obtain $B$. Then project the columns of $B$ onto the span of the first $k$ eigenvectors of $B$, and cluster based on a threshold on distance between the projected rows. We have the following theorem for its correctness. 

\begin{theorem}[Centered SVD algorithm]
\label{thm: centered-svd}
There are constants $C_0, C_1>0$ such that the following holds. 
Let $ p,q \le 0.75$ be parameters 
such that $\max \{p(1-p),q(1-q) \} \ge C_0(\log n)/n$. Let $G$ be a random graph sampled from $\USSBM(n,k,p,q)$. 
If $n \ge C_1 \cdot (\log n)^{14} \cdot p(1-q) \cdot k^2/(p-q)^2$
then the centered SVD (Algorithm~\ref{alg:oursvd}) recovers all the clusters with probability $1-\OO(1/n)$.
\end{theorem}
To the best of our knowledge, this is the first spectral algorithm without pre-SVD partition or post-SVD trimming steps that recover algorithm in the balanced case for $p,q=o(1)$ and $k=\omega (\log n)$. In fact, our parameters are comparable to that of SVD-II, which is a more complex (non-vanilla) algorithm. 
We finish this discussion with some remarks about our algorithm.

\begin{remark}[Notes on centering]
We further emphasize that even though we slightly deviate from SVD-I, our additional step of shifting is not synthetic. 
The step of shifting each entry of $A$ by $q$ is essentially a centering step.
Centering is a common step done in clustering algorithms, even beyond the $\SBM$ setup. 

For example, the algorithm PCA can be thought of as SVD projection of the centered matrix. 
Centering is done to remove the effect of the top eigenvector of the matrix, which often aligns close to the data mean and does not reveal knowledge about cluster identities. 
Our analysis of the centered-SVD algorithm thus provides insight towards the success of simple algorithms used in practice. 

Furthermore, $A$ and $B$ are only separated by a constant ($q$). Thus we believe an analysis of the eigenspace of $A$ is also possible in a similar manner. This remains an interesting problem.
\end{remark}

Now we describe the outline of our proofs.

\subsection{Proof Outlines}
\subsubsection{ Proof Outline of  Theorem~\ref{thm: main}  }
\label{sec: sketch-power}
We first revisit some notations. Recall that $A\sim\SSBM(n,k,p,q)$ is the adjacency matrix of the random graph, and $B:=A-q\cdot 1^{n \times n}$ is the matrix obtained by shifting each entry of $A$ by $q$. As described in Algorithm~\ref{alg:powermethod}, we aim to recover the largest cluster based on the $\ell_2$ distance of rows of the power matrix $B^r$ where $r=\log n$.

Specifically,  our goal is the following.  Let $v_i$ and $v_j$ be two vertices corresponding to the $i$-th and $j$-th row of $A$ (and thus, $B$).
We want to show that if $v_i$ and $v_j$ both belong to a large enough cluster then the distance of $i$-th and $j$-th row of $B^r$ is small.  On the other hand,  if $v_i$ belongs to a large enough cluster and $v_j$ belongs to a different cluster,  the distance between the rows of $B^r$ is much larger.  This allows us to recover a large cluster through distance based separation.

With this aim,  we decompose $B=L+R$ where $L:=\mathbb{E}[B]$ is the ``structure part'' of this matrix containing the hidden community information,
and $R$ is the ``noise part''. We notice that $R$ is a symmetric matrix where each entry is a $0$ mean independent random variable.

In this regard, we can write $B^r$ as $(L+R)^r$ and our aim is to show that
when considering the distance between two rows of $B^r$,  the distance due to $L^r$ is the dominating component. Here $L$ is a well structured matrix. By a simple calculation, for any two vertices $v_i$ and $v_j$,
\begin{itemize}
    \item If $v_i$ and $v_j$ belong to the same cluster, then the distance between $i$-th and $j$-th row of $L^r$ is $0$.

    \item Otherwise if $v_i\in V_{\ell}$ and $v_j\in V_{\ell'}$ 
    are from two different clusters, the distance between their rows
    is more than $\max \left \{ \sqrt{|V_{\ell}|} (p-q)^r |V_{\ell}|^{r-1}, \sqrt{|V_{\ell'}|} (p-q)^r |V_{\ell'}|^{r-1} \right \}$.
\end{itemize}

We then show that the $\ell_2$ norm of rows of 
$B^r-L^r$ is small (compared to the gap provided by $L^{r}$) by using ideas from random matrix theory and concentration bounds on low degree polynomials. 
We use a two step decomposition to achieve this.
In the first step we decompose
\[ B^r=(L+R)^r= L^r + \underbrace{L^{r-1}R+L^{r-2}RB+ \cdots LRB^{r-2} }_{M}+ RB^{r-1} \]
This is a decomposition based on the first location of $R$ in the product terms. 

\paragraph{Analysis of $M$.}We first show that the $\ell_2$ norm of rows of $M$ is small with high probability.  To achieve this, we start by showing that the absolute value of entries of $L^tR$ is sufficiently small (with high probability), which is fairly straightforward as the entries of $L^tR$ is just a linear combination of independent random variables $R_{i,j}$. This gives us an upper bound on the norm of rows of $L^tR$.  
We shall show that once this is upper bounded,  bounding the norm of rows of $L^tRB^{r-t-1}$ follows from some useful tools of random matrix theory~\cite{furedi1981eigenvalues,vu-norm}.

\paragraph{Analysis of $RB^{r-1}$.}Then, we are left with bounding the term $RB^{r-1}$.
This time we decompose based on the first location of $L$, obtaining
\[
RB^{r-1}= \underbrace{RLB^{r-2}+R^2LB^{r-3}+ \cdots R^{r-1}L}_{M'}+R^r
\]
We can view $R^r$ as a completely noisy component. One can also get a high level intuition behind the success of power method from the analysis of $R^r$.
According to condition of Theorem~\ref{thm: main}, we have that spectral norm of $L$ is much larger than that of $R$. However,
the $\ell_2$ norm of rows of $L$ corresponding to the largest cluster can still be smaller than the $\ell_2$ norm of rows of $R$.
But when we raise the matrix to its $r$-th power
for a large enough $r$, we will observe that rows of $L^r$ have much higher norm than that of $R^r$.
This shows that powering reduces the effect of noise on the $\ell_2$ norm of the rows.

\paragraph{Analysis of $M'$.}Finally, we upper bound norm of rows of $M'$,  for which we bound the absolute value of the entries of $R^tL$. This is our main technical contribution. We observe that each entry of this matrix is a low degree ($t \le \log n$) polynomial of the random variables $R_{i,j}$ where there is correlation between the different monomials. We use a random partition based analysis to bound the entries, and this is arguably the most non-trivial step in our analysis.

\paragraph{Analyzing entries of $R^tL$ via random partition.}We can write down the 
$(a,b)$-th entry of $R^tL$ (where $1 \le a,b \le n$)
as
\[ 
(R^tL)_{a,b}= \sum\limits_{(\ell_1, \ldots ,\ell_t) \in [n]^t} R_{a,\ell_1}R_{\ell_1,\ell_2}\cdots R_{\ell_{t-1},\ell_t}L_{\ell_t,b}
\]

Then $(R^tL)_{a,b}$ is a sum of multivariate monomials, where each monomial is uniquely defined by an index list $\Ell=(\ell_1, \ldots ,\ell_t)$. Against this setup, we divide the lists $\Ell$ into several groups through an encoding (that we describe in Section~\ref{sec:proof}) into a $t$-length array $X$ based on duplicacy in the entries of the index lists. All the index lists that are mapped into the same encoding $X$ are assigned into a group $\zeta_X$. Let the set of all $X$ be $\mX$.
Then we have 
\[
(R^tL)_{a,b}=
\sum_{X \in \mX}
\underbrace{ \left(
\sum_{(\ell_1, \ldots ,\ell_t) \in \zeta_X} R_{a,\ell_1}R_{\ell_1,\ell_2}\cdots R_{\ell_{t-1},\ell_t}L_{\ell_t,b}
\right)}_{\Z(X)}
\]
We show that $|\mX|\le t^t$ and then upper bound 
$\abs{\Z(X)}$ for any arbitrary $X$, and complete the proof through an union bound.

For the purpose of explanation, here we focus on the group $\zeta_{X'}$ such that any index list $(\ell_1, \ldots ,\ell_t) \in \zeta_{X'}$ is made of $t$ many different values. As a consequence, $\Z(X')$ is a multilinear polynomial. As we shall see, this fact simplifies our analysis. 
Note that $\Z(X')$ is a ``sum of correlated product terms'' (SOP), as the same random variable $R_{i,j}$ can occur in many monomials. 

Against this setup, we layout our random partition idea.
Let $T$ be a random partition of $\{1, \ldots ,n\}$ into $t$ many sets $T_1, \ldots ,T_t$
where each element in $[n]$ is assigned to one of the $t$ many sets uniformly at random.
Then we denote
\[
|S_T|:=\abs{\sum\limits_{\ell_1 \in T_1, \ldots ,\ell_t \in T_t} R_{a,\ell_1}R_{\ell_1,\ell_2}\cdots R_{\ell_{t-1},\ell_t}L_{\ell_t,b}}
= \Bigg| \sum\limits_{\ell_1 \in T_1 ,\ldots  ,\ell_{t-1} \in T_{t-1}}
R_{a,\ell_1}R_{\ell_1,\ell_2}\cdots R_{\ell_{t-2},\ell_{t-1}}
\underbrace{\left( \sum_{\ell_t \in T_t} R_{\ell_{t-1},\ell_t}L_{\ell_t,j}\right)}_{F(t-1,\ell_{t-1})}\Bigg|
\]
We first observe that for each 
$\ell_{t-1}$, $F(t-1,\ell_{t-1})$ is a summation of independent random variables with range in $[-1,1]$. 
Hence we can apply a Chernoff bound on random variables $R_{\ell_{t-1}, \ell_{t}}$ where $\ell_{t-1}\in T_{t-1}$ and $\ell_t\in T_t$ to upper bound $|F(t-1,\ell_{t-1})|$. Assume that we upper bound this quantity by $\beta$ with high probability.
We can carry out this process recursively to analyze $S_T$. Precisely, we have 
\[
|S_T|= \Bigg| \sum\limits_{\ell_1 \in T_1 ,\ldots  ,\ell_{t-2} \in T_{t-2}}
R_{a,\ell_1}R_{\ell_1,\ell_2}\cdots R_{\ell_{t-3},\ell_{t-2}}
\underbrace{\sum_{\ell_{t-1} \in T_{t-1}} R_{\ell_{t-2},\ell_{t-1}}F(t-1,\ell_{t-1}}_{F(t-2,\ell_{t-2})})
\Bigg|
\]
We want to upper bound $\abs{F(t-2,\ell_{t-2})}$ in a similar manner. Here observe that since $R_{\ell_{t-2},\ell_{t-1}}$ is a zero mean random variable, so is $R_{\ell_{t-2},\ell_{t-1}}F(t-1,\ell_{t-1})$, as $F(t-1,\ell_{t-1})$ is a constant for any $\ell_{t-1} \in T_{t-1}$. This is because we have already sampled (fixed) the random variables $R_{\ell_{t-1},\ell_t}$.

Furthermore, as $|R_{\ell_{t-2},\ell_{t-1}}\cdot F(t-1,\ell_{t-1})|\leq\beta$ with high probability, we can then apply Chernoff bound to upper bound $\abs{F(t-2, \ell_{t-2})}$. We keep repeating this step recursively and bound $|S_T|=\abs{F(0,\ell_0)}$ with high probability for any fixed $T$. The crux here is that $T_1,\dots,T_{t}$ are disjoint, so we can 
iteratively apply the Chernoff bound $t$ times.

However, for any fixed partition $T$, $S_T$ does not cover all monomials in $\Z(X')$. But as $T$ is designed uniformly at random, any monomial in $\Z(X')$ has an equal and a relatively large chance of being represented in $S_T$.
We use this fact to show $|\Z(X')| \le f(t) \cdot \mathbb{E}_T[|S_T|]$ for a reasonably small $f(t)$. 

We have so far bound $|S_T|$ with high probability only for any fixed $T$. What we want is to bound $\mathbb{E}[|S_T|]$. Here taking a union bound on all choices of $T$ is not a possibility due to the fact that there can be exponentially many such sets $T$.  Instead we use Markov's higher moment ideas to obtain bounds, which completes our proof.  

We remark that the intuition behind our encoding scheme is to generalize the conversion from SOP into POS beyond the simple case described above, which is aided by our scheme.  
Next, we consider the summation $\Z(X)$ for a generic $X \in \mX$. 
Here we encounter sums of monomials where individual variables may have degree more than $2$ in the monomials. We use similar arguments as the case above to bound the sum for all such groups. 
Then taking an union bound on all $\Z(X)$ upper bounds $|(R^tL)_{a,b}|$. Again, once the absolute value of the entries of $R^tL$ is bounded, an upper bound on the norm of rows of $R^tLB^{r-t-1}$ follows. 

The complete proof of the random partition analysis can be found in Section~\ref{sec:proof}. Next we describe the outline of the proof for our centered SVD algorithm, which builds up on the power iteration.  

\subsubsection{ Outline of Theorem~\ref{thm: centered-svd} }
We use our analysis in the spectral algorithm for the $\USSBM$ model. The key observation is as follows. Let $P^k_B$ be the SVD projection operator onto the first $k$ eigenvectors of $B$. Let $\bu$ and $\bv$ be the $i$-th and $j$-th column of $B$. Then we show that the distance between $P^k_B\bu$ and $P^k_B\bv$ is well approximated by the distance between the $i$-th and $j$-th row of $B^{r+1}$. Then the result follows from the power iteration method.
This analysis implies that in the centered case, the partition step for reducing correlation described in SVD-II is indeed a technical step, and is not a necessity (at least for the nearly balanced case).
To the best of our knowledge, this is the first analysis of a ``simple'' spectral algorithm that is applicable for $p,q=o(1)$ and $k=\omega(\log n)$. 
Next, we discuss some connections between our result and existing works.

\subsection{Connections to Other Techniques}
\label{sec: connections to others}

In this section, we discuss connections between our analysis and some standard techniques in random matrix theory,
and the challenges therein. 
We believe that our method can be helpful in resolving some barriers in previous analysis.

\paragraph{An analogue of Wigner's trace method.}In random matrix theory, it is very important to understand the distribution of eigenvalues (or singular values) of random matrices. However, it is always not easy to analyze eigenvalues directly. To solve this question, Wigner \cite{wigner1958distribution} proposed to use the following equality:
\[
\sum_{i}\lambda_{i}(A)^{k} = \mathbf{Trace}(A^k),
\]
where $\lambda_{i}$ is the $i$-th singular value of $A$. Then this question can be translated to the analysis of the diagonal entries (trace) of the power matrix $A^k$. Please see \cite{furedi1981eigenvalues, vu-norm} for more details. The entries of $A^k$ is a more intuitive notion than eigenvalues and we have combinatorial tools, such as the tree encoding scheme \cite{furedi1981eigenvalues, vu-norm}, to analyze it. However, in many applications (including but not limited to analysis of vanilla spectral algorithms),
one needs to analyze the eigenvectors of $A$, and/or the eigenspace formed by top eigenvectors of $A$. 
In this direction, our idea can be considered as a link between entries of power matrices and projection onto eigenspace for random matrices, 
as we analyze eigenspace projection (Theorem \ref{thm: centered-svd}) through entry-wise analysis (not only diagonal entries) of $A^k$.

Next, we discuss some of the standard tools used in the analysis of eigenvectors and eigenspaces, along with their limitations and potential ways forward.

\paragraph{Beyond Davis-Kahan.}The SBM problem can be thought as a ``structure+noise'' problem, where the mean ($\mathbb{E}[A]$) contains the hidden partition and $R:=A-\mathbb{E}[A]$ is the random perturbation (noise). 
One of the core ideas behind the spectral algorithms 
is to show that the eigenspaces due to the top eigenvectors 
of $A$ and $\mathbb{E}[A]$ are similar.
In this direction, the 
Davis-Kahan $\sin(\Theta)$ theorems~\cite{Davis-Kahan} are arguably the most standard way of bounding the distances between these eigenspaces. However, although Davis-Kahan is tight when the perturbation is chosen adversarially, it is sometimes sub-optimal when the perturbation is random, as is the case of the SBM problem. Then one needs to either use Davis-Kahan $\sin(\Theta)$ more subtly, or avoid it altogether. Please see \cite{Eldrige-Unperturbed} for more discussions about Davis-Kahan and SBM. 
In this direction, Eldridge et al.~\cite{Eldrige-Unperturbed} made an effort to analyze SBM beyond Davis-Kahan (a part of their proof also used power iteration). However, their analysis seems complicated and hard for large $k$ (the number of clusters). 
Furthermore, Abbe et al.~\cite{abbe2020entrywise} also indicated some vague connections between power iteration and Davis-Kahan. They comment,
\begin{quote}
\emph{Nevertheless, we believe the iterative perspective is helpful to many other (nonconvex)
problems where a counterpart of Davis-Kahan theorem is absent.}
\end{quote}

To summarize, power iteration based techniques seems to be promising in analysis of many applications beyond the limits of Davis-Kahan.

\paragraph{Beyond individual eigenvector analysis.}
Although the aforementioned works improved on Davis-Kahan bounds, 
a common theme in both of the papers
is analysis of the individual eigenvectors. 
Specifically, they analyze the distance between individual eigenvectors of the expectation and perturbed matrices.
However, if the expectation matrix $\mathbb{E}[A]$ has high duplicity for any one eigenvalue (which is the case with the balanced SBM model, $\lambda_2(\mathbb{E}[A]) = \cdots=\lambda_k(\mathbb{E}[A])$), either the methods do not apply directly, or there is an adversarial cost in the bounds depending on this duplicity.

In comparison, we do not analyze individual eigenvectors. Our analysis studies the whole eigenspace generated by top eigenvectors.
We note that McSherry \cite{mcsherry2001spectral} and then Vu~\cite{vu2018simple} (SVD-II) also analyzed whole
eigenspace. Hence, they do not have the multiple eigenvalue barrier. However, they needed to use a partition step in their algorithm to analyze it, to reduce correlation. In contrast, we analyze the entries of power matrices, which addresses the technical barriers mentioned above.

\subsection{Conclusion and Future Directions}

In this paper we present a simple algorithm based on the power iteration method to recover communities in the $\SSBM$ model. Our algorithm works in the presence of small clusters, 
with recovery conditions being logarithmically 
close to the state of the art~\cite{vu2018simple,mukherjee2022recovering}.

Next, we connect SVD projection to the power iteration method in the (almost) balanced case, and obtain  a vanilla spectral algorithm for $k=\omega( \log n)$ and $p,q=o(1)$, 
with recovery guarantee comparable to that of Vu~\cite{vu2018simple}.

The central idea behind analysis of entries of power matrices of $B$ (the shifted adjacency matrix) is the random partition technique. We analyze entries of powers of the matrices as low degree polynomials, that can be seen as a sum of correlated monomials. We convert these into product of sums of lower-degree uncorrelated monomials. This perspective allows us to improve on restrictions of related existing works, such as the pre-projection partition of Vu~\cite{vu2018simple}, and need for number of clusters to be constant (\cite{Eldrige-Unperturbed}) among others.
Furthermore, we believe these tools can be of independent interest in analysis of the ``structure+noise'' model.

\paragraph{Challenges.}
In the current form we can only establish the link between the power iteration method and SVD projection when the top $k$ eigenvectors of $A$ have similar eigenvalues. This is the reason we only focus on the uniform $\SSBM$ model. Extending this method for a more general case remains a very interesting research direction. 

Another question of interest is replacing the shifting procedure in the power iteration method. Analyzing whether a shifting procedure based on $A$
(instead of knowledge of $q$) works is an interesting vertical, making the algorithm more practical.

\section{Preliminaries}
\label{sec-2}

Here we define the notations and definitions we use throughout the paper. We start with spectral properties of the random matrices related to the graph generated through $\SSBM$.

\paragraph{Eigenvalues of the (centered) adjacency matrix}

Recall that we are given the random adjacency matrix $A$ that is sampled according to the underlying $\SSBM$ framework. 
In this paper we work with the centered adjacency matrix $B$, which is obtained by subtracting $q$ from every entry of $A$. The expectation matrix of $B$, denoted as $\meanB$ is thus a block diagonal matrix, with each entry being $p-q$ or $0$. Then we have that $\lambda_i(\meanB)=(p-q)|V_i|, 1 \le i \le k$. The primary difference in our algorithm from SVD I is that we consider the matrix $B$ instead of $A$. Now, we define some norm notations, followed by results on spectral norms of symmetric matrices.

\begin{definition}[Norm-related notations]
\label{def: M-norm}
For any matrix $M$, we denote by $M_i$ the $i$-th row of $M$.
In this direction, for a matrix $M$,
\begin{itemize}     
    \item 
    $\|M\|$ denotes the spectral norm of the matrix. 
    
    \item $\|M_i\|$ denotes the $\ell_2$-norm of the $i$-th row of $M$ 

    \item $\|M\|_{row}$ denotes the maximum of the $\ell_2$-norm of the rows of $M$.
\end{itemize}

\end{definition}

\begin{lemma}[Norm of a random matrix \cite{vu2018simple}]
\label{lemma: random-matrix-norm}
There is a constant $C_0>0$ such that the following holds. Let $E$ be a symmetric matrix whose upper diagonal entries $e_{ij}$ are independent random
variables where $e_{ij} = 1-p_{ij}$ or $-p_{ij}$ with probabilities $p_{ij}$ and $1-p_{ij}$ respectively, where $p_{ij}\in[0,1]$. Let $\sigma^2 := \max_{ij} \{p_{ij}(1-p_{ij})\}$. If $\sigma^2\geq C_0 \log n/n$, then
\[
\Pr[\|E\|\geq C_0\sigma n^{1/2} ] \leq n^{-3}
\]
\end{lemma}

\begin{lemma}[Weyl's inequality]
Let $A=\meanA+E$ be a matrix. Then $\lambda_{t+1}(A) \le \lambda_{t+1}(\meanA) + \| E \|$.
\end{lemma}

We have the following simple corollary.

\begin{corollary}
\label{corollary: spectrum}
Let $A=\meanA+E$ be the empirical adjacency matrix. Then with probability at least $1-n^{-3}, $we have that
\[
\lambda_{i}(\meanA)-C_0\sigma n^{1/2} \leq \lambda_{i}(A) \leq  \lambda_{i}(\meanA)+C_0\sigma n^{1/2} 
\]
for every $i\in [n]$.
\end{corollary}

Next we write down the Chernoff-Hoeffding bound for small probability events. 

\begin{theorem}[Chernoff Hoeffding bound~\cite{Chernoff}]
\label{thm: chernoff}
Let $X_1, \ldots , X_n$ be i.i.d random variables that can take values in $\{0,1\}$, with $\E[X_i]=p$ for $1 \le i \le n$. 
Then we have 
\begin{enumerate}
    \item $\Pr \left( \frac{1}{n}\sum_{i=1}^n X_i \ge p+ \eps \right) \le e^{-D(p+\eps||p)n} $    

    \item $\Pr \left( \frac{1}{n}\sum_{i=1}^n X_i \le p - \eps \right) \le e^{-D(p-\eps||p)n} $

\end{enumerate}

\end{theorem}

We have now also presented the notions needed to define a suitable constant $C$ in Theorem~\ref{thm: main}.

\begin{definition}[The constant $C$]
\label{def: Constant-C}
We define $C:=C_0 \cdot 10^7$ where $C_0$ is the constant referred in Lemma~\ref{lemma: random-matrix-norm}. 
\end{definition}
Thus, under the condition of Theorem~\ref{thm: main}
we have $(p-q)\smax \ge 10^6 \cdot C_0 \cdot \sqrt{p(1-q} \cdot \sqrt{n} \cdot (\log n)^7$.

We conclude this section by defining some more notations and assumptions that we use throughout the paper.
\begin{enumerate}

    \item We denote the vertices of the graph with the corresponding row/column. For two integers $i$ and $j$, we write $i \sim j$ if the $i$-th and $j$-th vertices belong to the same cluster, and $i \nsim j$ otherwise.

    \item For the cluster $V_{\ell}$, we denote the size of the cluster by $s_{\ell}$. Furthermore, 
    by $\smax$ and $\smin$ we denote the size of the largest and the smallest cluster respectively.

\end{enumerate}

\section{Analysis of the Power-iteration Method}
\label{sec:prove-power}

In this section, we prove Theorem~\ref{thm: main}. Recall that $A\sim\SSBM(n,k,p,q)$ is the adjacency matrix and $B:=A- q\cdot 1^{n \times n}$ is the centered matrix. Let $V_1$ be the largest hidden cluster with $\smax=|V_1|$. 
Theorem \ref{thm: main} aims to show that there exists some threshold $\Delta$ such that with high probability for any $v_i$ and $v_j$,
\begin{enumerate}
    \item If $v_i,v_j\in  V_1$ then $\|B^r_i-B^r_j\| \le \Delta/2$.
    \item If $v_i\in V_1$ and $v_j\not\in V_1$, then $\|B^r_i-B^r_j\| \ge \Delta$.
\end{enumerate}
This separation shows that Algorithm \ref{alg:powermethod} is able to detect $V_1$ correctly, which finishes the proof Theorem \ref{thm: main}. 
We choose $\Delta:=0.5\sqrt{\smax}(p-q)^r(\smax)^{r-1}$ for Algorithm~\ref{alg:powermethod}.

Continuing from Section \ref{sec: sketch-power}, we decompose $B=L+R$ where $L=\mathbb{E}[B]$ is the structure part and $R$ is the noise part. Hence, $B^r$ can be decomposed as 
\begin{equation}
\label{eq: decompose}
B^r=L^r+M+M'+R^r
\end{equation}
where $M=L^{r-1} R +  L^{r-2} R B + \dots +  LR  B^{r-2}$ and $M'=RLB^{r-2}+ R^2LB^{r-3}+ \cdots + R^{r-1}L$. 

As we discussed in Section \ref{sec: sketch-power}, we would like to show that $L^r$ is the dominant part in $B^r$. That is, $\|B_{i}^{r}-B_{j}^{r}\|\approx \|L_{i}^{r}-L_{j}^{r}\|$ for every $i,j$. From Equation~\ref{eq: decompose} we have for any $i,j$
\begin{align*}
&\left| \|B^r_i-B^r_j\|-\|L^r_i-L^r_j\| \right| \\
\leq&\|(B^r_i - L^r_i) - (B^r_j - L^r_j)\| \\
\le &\|M_i-M_j\|+\|M_i'-M_j'\|+\|R^r_i-R^r_j\| \\
\leq & 2(\|M\|_{row}+\|M'\|_{row}+\|R^r\|_{row})
\end{align*}

The proof then progresses in the following manner.

\begin{itemize}
    \item Lemma~\ref{lemma: norm:L} first shows there is separation on $L$. That is, for every $v_i,v_j\in V_1$, $\|L_{i}^{r} - L_{j}^{r}\|=0$; for every $v_i\in V_1$ and $v_j\not\in V_1$, $\|L_{i}^{r} - L_{j}^{r}\|\geq 2\Delta$
    \item Corollary~\ref{cor: norm-M} shows that $\|M\|_{row} \le 0.1\Delta$ with probability $1-\OO(1/n)$.
    \item Lemma~\ref{lemma :norm: Rr} shows that 
    $\|R^r\|_{row} \le 0.1\Delta$ with probability $1-\OO(1/n)$.
    \item Corollary~\ref{cor: norm: M'} shows that $\|M'\|_{row} \le 0.1\Delta$ with probability $1-\OO(1/n)$.
\end{itemize}

Then we get that if $v_i \in V_1$ with $|V_1|=\smax$, for any $i \sim j$ we have
with probability $1-\OO(1/n)$, $\|B^r_i-B^r_j\| \le 0.6\Delta$. On the other hand, if $i \nsim j$ then $\|B^r_i-B^r_j\| \ge 2\Delta-0.6\Delta \ge 1.4\Delta$. This completes the proof.

Now we obtain these bounds. The conditions on $\smax$ and $p,q$ are the same as in Theorem~\ref{thm: main} in all of the following lemmas/corollaries, and thus we do not repeat them in each statement.

\subsection{Bounding the Maximum Norm of the Rows}

We start by recording some facts about $L$ and $R$.
\begin{fact}
\label{fact: L}
Let $L,R$ be the matrices defined above. We have the following properties for them.
\begin{enumerate}
    \item For every $r\geq 1$ and $i, j\in V_{\ell}$, $L^{r}_{i,j}=(p-q)^r\cdot s_{\ell}^{r-1}$
    \item For every $r\geq 1$ and $i\not\sim j$, and $L^{r}_{i,j}=0$

    \item $R_{i,j}$ are zero mean independent random variables (modulo the symmetry) 
    \item For every $i,j\in V_{\ell}$, $R_{ij} = 1-p$ or $-p$ with probabilities $p$ and $1-p$ respectively. For every $i\not\sim j$, $R_{ij} = 1-q$ or $-q$ with probabilities $q$ and $1-q$ respectively.

    \item The maximum variance of $R_{i,j}$, $\sigma^2$ is upper bounded by $p(1-q)$.
\end{enumerate} 
\end{fact}

Now we analyze $\|L^{r}_i - L^r_j\|$, obtaining the following result.
\begin{lemma}
\label{lemma: norm:L}
Let $L$ be the matrix defined above.
Let $v_i \in V_1$. Then for every $i\sim j$, $\|L_{i}^{r} - L_{j}^{r}\| =0$.
Otherwise if $i \nsim j$ then 
$\|L_{i}^{r} - L_{j}^{r}\| \ge 2\Delta$.
\end{lemma}
\begin{proof}
If $i \sim j$, then according to Fact \ref{fact: L}, the $i$-th and $j$-th row are same in $L^r$, and thus $\|L^r_i-L^r_j\|=0$.

On the other hand, if $i \nsim j$, then the $i$-th and $j$-th row of $L^r$ differ in at least $\smax$ positions because $v_i\in V_1$. If we consider the $i$-th row, there are $\smax$ positions where $(L^r)_{i,\ell}=(p-q)^r(\smax)^{r-1}$ and $(L^r)_{j,\ell}=0$. This implies
$\|L^r_i-L^r_j\| \ge \sqrt{\smax} \cdot (p-q)^r(\smax)^{r-1} \ge 2\Delta$.
\end{proof}
We spend rest of the section upper bounding
the maximum $\ell_2$ norm of rows of $M, M'$ and $R^r$ with high probability. 
We start by upper bounding $\|B\|$. 
\begin{lemma}
\label{lemma: norm of B}
Let $L$ and $R$ be the matrices defined above. Then
\[
\Pr[\|B\|\geq (p-q)\smax + C_0\sigma n^{1/2} ] \leq n^{-3}
\]
\end{lemma}
\begin{proof}
We recall that $B = L+R$. Hence, $\|B\|\leq \|L\|+\|R\|$. By Fact \ref{fact: L}, we also have that $\|L\| = (p-q)\smax$, which implies
\[
\Pr[\|B\|\geq (p-q)\smax + C_0\sigma n^{1/2} ] \leq \Pr[\|R\|\geq C_0\sigma n^{1/2} ] \leq n^{-3}
\]
where the second inequality follows from Lemma \ref{lemma: random-matrix-norm}.
\end{proof}
Now we observe a simple and useful property on the $\ell_2$ norm of product of matrices, that follows from the definition of spectral norm.

\begin{fact}
\label{fact: row-norm-product}
For any $n \times n$ matrices $B$ and $C$,
\[
\|B\cdot C\|_{row} \le \|B\|_{row} \cdot \|C\|
\]
\end{fact}

\paragraph{Bounding the rows of $M$.}

Recall that $M=L^{r-1}R + L^{r-2}RB+ \dots + LR\cdot B^{r-2}$. To bound each of the individual term, we first look at the matrix $L^{t}R$ for $1 \le t \le r-1$.

\begin{lemma}
With probability at least $1-n^{-3}$, the absolute value of all entries of $L^{t}R$ is upper bounded by $96\cdot (\sqrt{p(1-q)}\cdot \sqrt{\smax}\cdot \log n)\cdot (p-q)^{t}\cdot (\smax)^{t-1}$ for $t \le \log n$.
\end{lemma}
\begin{proof}
For each pair $(i,j)$, the $(i,j)$-th entry of $L^{t}R$ is
\[
\sum_{\ell} L^{t}_{i,\ell}\cdot R_{\ell,j}
\]
Let the $i$-th vertex belong to $V_a$ for some $a$. Then by Fact \ref{fact: L}, we have that 
\[
\sum_{\ell} L^{t}_{i,\ell}\cdot R_{\ell,j}  = \sum_{\ell\in V_a} L^{t}_{i,\ell}\cdot R_{\ell,j} + \sum_{\ell\not\in V_a} L^{t}_{i,\ell}\cdot R_{\ell,j} =  (p-q)^{t}\cdot (|V_a|)^{t-1}\cdot \sum_{\ell \in V_{a}} R_{\ell,j}
\]
We notice that $\sum_{\ell \in V_{a}} R_{\ell,j}$ is a summation of at most $\smax$ many zero-mean independent random variables. Then by the Chernoff bound (Theorem \ref{thm: chernoff}), we have that
\[
\Pr\left[\left|\sum_{\ell \in V_{a}} R_{\ell,j}\right| \geq 48\sqrt{p\cdot \smax}\cdot \log n\right] \leq n^{-5}
\]
Then the absolute value of $(i,j)$-th entry of $L^tR$ is upper bounded by 
$(p-q)^{t}\cdot (\smax)^{t-1} \cdot 48\sqrt{p} \cdot \sqrt{ \smax}\cdot \log n$ with probability $1-n^{-5}$. Here we use that $1 \ge  \sqrt{1-q} \ge 1/2$ and then taking an union bound on all $n^2$ $(i,j)$-pairs completes the proof.
\end{proof}

We use this result to upper bound $\|M\|_{row}$.

\begin{corollary}
\label{cor: norm-M}
Let $r= \log n$ and $\Delta=0.5\sqrt{\smax}(p-q)^r(\smax)^{r-1}$.
Then the $\ell_2$ norm of each row of the matrix
$M=L^{r-1}R + L^{r-2}RB \dots + LR\cdot B^{r-2}$
is upper bounded by $0.1 \Delta$ with probability $1-\tilde{\OO}(n^{-3})$. 
\end{corollary}
\begin{proof}

We have with probability $1-\OO(n^{-3})$,
all entries of $L^tR$ are upper bounded by

\[
h_t:=96\cdot (\sqrt{p(1-q)}\cdot \sqrt{\smax}\cdot \log n)\cdot (p-q)^{t}\cdot (\smax)^{t-1}.
\]

Thus, the $\ell_2$-norm of any row of $L^tR$ is upper bounded by $h_t\sqrt{n}$. 
Then from Fact~\ref{fact: row-norm-product} 
we have $\|L^tRB^{r-t-1}\|_{row} \le  \|L^tR\|_{row} \cdot \|B\|^{r-t-1}$.
Substituting $\|B\|$ from Lemma~\ref{lemma: norm of B} we get 
\[
\|L^tRB^{r-t-1}\|_{row} 
\le  h_t\sqrt{n} \cdot \left( (p-q)\smax + C_0\sigma n^{1/2} \right)^{r-t-1}
\]
with probability $1-\OO(n^{-3})$. Then the fact that $(p-q)\smax \ge (\log n)^7 C_0 \sigma \sqrt{n}$ and $r-t-1 \le \log n$ gives us the bound
\[
\|L^tRB^{r-t-1}\|_{row}  \le (192 \sqrt{p(1-q)}\sqrt{n}\log n). 
\cdot \sqrt{\smax} \cdot 
(p-q)^{r-1} (\smax)^{r-2}
\]
Since $M=L^{r-1}R + L^{r-2}RB+ \dots + LR\cdot B^{r-2}$ we have
\[
\|M\|_{row} \le \|L^{r-1}R\|_{row} + \|L^{r-2}RB\|_{row} + \dots + \|LRB^{r-2}\|_{row}.
\]
Then by a union bound, we have with probability $1-\OO(r \cdot n^{-3})$
\[
\|M\|_{row}
\le r \cdot (192 \sqrt{p(1-q)}\sqrt{n}\log n)
\cdot \sqrt{\smax} \cdot 
(p-q)^{r-1} (\smax)^{r-2}
\]
Substituting $r \cdot (192 \sqrt{p(1-q)}\sqrt{n}\log n) < 0.1(p-q)\smax$
then completes the proof.

\end{proof}

\paragraph{Bounding the Rows of $R^r$.} Now we upper bound rows of $R^r$.

\begin{lemma}
\label{lemma :norm: Rr}
With probability $1-\OO(n^{-3})$, the $\ell_2$ norm of all rows of $R^{r}$ is upper bounded by 
$0.1 \Delta$ for $r=\log n$.
\end{lemma}
\begin{proof}

We know that $\|R^r\|_{row} \le \|R^r\|$.
Combined with Lemma \ref{lemma: random-matrix-norm} this implies that with probability $1-\OO(n^{-3})$, 
\[
\|R^r\|_{row} \le
\|R^r\| \le \|R\|^r \le (C_0 \sqrt{p(1-q)} \cdot \sqrt{n})^r.
\]
Now we use $(p-q)\smax \ge 10000C_0 (\log n)^7 \sigma \sqrt{n}$ to get the following inequality. 
\[
\|R^r\|_{row} \le (C_0 \sigma \sqrt{n})^r
\le \left( \frac{(p-q)\smax}{(\log n)^7} \right)^r \le 0.1\Delta
\]
\end{proof}

\paragraph{Bounding the rows of $M'$.}
Here we want to bound the norm of rows of $M'= RLB^{r-2} + \ldots +R^{r-1}L$. 
To bound the norm of rows of $R^tLB^{r-t-1}$, we start with bounding the norm of rows of $R^tL$. In this direction we obtain the following result.

\begin{lemma}
\label{lem:last-step}
The absolute value of each entry of $R^{t}L$ is upper bounded by 
\[
C_2 \sqrt{p(1-q)}(\log n)^6
\cdot  ((p-q) \sqrt{\smax} ) ((p-q)\smax)^{t-1}
\]
for $t \le \log n$ with probability $1-\OO(n^{-3})$ for some constant $C_2 < 0.05C$. 
\end{lemma}
\noindent Due to the non trivial and technical nature of the proof, we place it in Section~\ref{sec:proof}. Finally, we use the aforementioned result to bound $\|M'\|_{row}$.

\begin{corollary}
\label{cor: norm: M'}
The maximum $\ell_2$-norm of the rows of $M'$ is upper bounded by $\|M'\|_{row} \le 0.1\Delta$ 
with probability $1-\tilde{\OO}(n^{-3})$.
\end{corollary}
\begin{proof}
We have $M'= RLB^{r-2} + \ldots +R^{r-1}L$.
This implies $\|M'\|_{row} \le \|RLB^{r-2}\|_{row} + \ldots +\|R^{r-1}L\|_{row}$. Lemma \ref{lem:last-step} gives upper bounds for each entries of $R^tL$ with probability $1-\OO(n^{-3})$. Then we have 
\[
\|R^tL\|_{row} \le  \sqrt{n}\cdot \left(C_2\sqrt{p(1-q)} (\log n)^6 \cdot (p-q)\sqrt{\smax} \cdot ((p-q)\smax)^{t-1}\right)
\]
with probability $1-\OO(n^{-3})$.
Furthermore $\|R^tLB^{r-t-1}\|_{row} \le \|R^tL\|_{row} \cdot (\|B\|)^{r-t-1}$.
This implies that with probability $1-\OO(n^{-3})$,
\[
\|R^tLB^{r-t-1}\|_{row} 
\le
2C_2 \sqrt{p(1-q)} (\log n)^6 \sqrt{n} \cdot 
\sqrt{\smax} \cdot (p-q)^{r-1} (\smax)^{r-2}
\]
for all $1 \le t \le r-1$. 
Summing over the $r-1$ terms we get with probability $1-\OO(\log n \cdot n^{-3})$,
\[
\|M'\|_{row} \le 2C_2 \sqrt{p(1-q)} (\log n)^7 \sqrt{n} \cdot 
\sqrt{\smax} \cdot (p-q)^{r-1} (\smax)^{r-2}
\]

Substituting  $2C_2 \sqrt{p(1-q)} (\log n)^7 \sqrt{n} \le 0.1 (p-q)\smax$
results in the bound
$ \|M'\|_{row} \le  0.1 \sqrt{\smax} (p-q)^r (\smax)^{r-1} \le 0.1\Delta $
with probability $1-\OO(\log n \cdot n^{-3})$. 

\end{proof}

This completes our proof of Theorem~\ref{thm: main}.

\section{An Analysis of Simple Spectral Algorithms}
\label{sec: spectral}

In this section we use the analysis of the power iteration method (Algorithm~\ref{alg:powermethod}) to prove the correctness of the centered-SVD algorithm (Algorithm~\ref{alg:oursvd}) and thus prove Theorem~\ref{thm: centered-svd}. 
Let us first briefly recall the algorithm. 
We are given the random adjacency matrix $A$ from $\USSBM(n,k,p,q)$, where the $k$ hidden communities are formed by putting each vertex into the communities uniformly at random (this implies that the communities are fairly balanced). Let $P^k_B$ the projection operator onto the eigenspace of the top $k$ eigenvectors of the shifted adjacency matrix $B$. 
Let $v_i$ and $v_j$ be the $i$-th and $j$-th vertices of $G$  
and correspondingly the $i$-th and $j$-th row (column) of $B$ be $\bu$ and $\bv$ respectively. 
We intend to show with high probability that there is a threshold $\Delta$ so that if $v_i$ and $v_j$ belong to the same community then $\|P^k_B\bu-P^k_B\bv\| \le \Delta$.
Otherwise if they belong to different communities then $\|P^k_B\bu-P^k_B\bv\| \ge 1.1\Delta$.
As we focus on the more specific $\USSBM$ case here, we show that this recovery works for all clusters (in comparison to Algorithm~\ref{alg:powermethod} where we consider recovering the largest cluster).

Furthermore we work under the condition that 
$n \ge C^2\cdot k^2 (\log n)^{14} \cdot p(1-q)/(p-q)^2$.
We now give a sketch of the proof. It can be broken down into two steps. 

\paragraph{Sketch of proof}

\begin{itemize}

\item 
First we show that in the $\USSBM$ case under the condition of Theorem~\ref{thm: centered-svd}, all communities are separable using the power analysis method. That is, for $r=\log n$ and a threshold $\Delta':=0.5 \sqrt{n/k} (p-q)^{r+1} (n/k)^{r}$
we show in Lemma~\ref{lem: balanced B} that with probability $1-\OO(1/n)$ if two vertices $v_i$ and $v_j$ belong to the same community then $\|B^{r+1}_i-B^{r+1}_j \| \le 0.7\Delta'$. Otherwise if they belong to different communities then $\|B^{r+1}_i-B^{r+1}_j \| \ge 1.2\Delta'$.

\item
Next we connect $P^k_B$ with $B^r$. We show that if $\bu$ is the $i$-th column of $B$, then $P^k_B\bu$ is well approximated by
$\frac{1}{(p-q)^r(n/k)^r} \cdot B^{r+1}_i$. In fact, we prove 
in Lemma~\ref{lem:power-SVD-distance} that
$\|P^k_B\bu - \frac{B^{r+1}_i}{(p-q)^r(n/k)^r} \| =o 
\left( \frac{\Delta'}{(p-q)^r(n/k)^r} \right)$ with probability $1-\OO(n^{-2})$.

\end{itemize}

From hereon we denote $s:=n/k$, $f_r:=\frac{1}{(p-q)^r(s)^r}$ and define the threshold of Algorithm~\ref{alg:oursvd} as $\Delta:=f_r\cdot \Delta'=0.5(p-q)\sqrt{s}$.
Then combining the results above we get the following.

For any two vertices $v_i$, $v_j$ and corresponding columns/rows $\bu$ and $\bv$,
\begin{enumerate}
    \item If $v_i,v_j \in V_{\ell}$ for some $\ell$ then by the triangle inequality,
    \[
    \|P^k_B\bu-P^k_B\bv \|\le     
    \|P^k_B\bu-f_rB^{r+1}_i \|+ \|f_rB^{r+1}_i-f_rB^{r+1}_j\|+\|P^k_B\bv-f_rB^{r+1}_j \|
    \]
    
    From Lemma~\ref{lem: balanced B} with  probability $1-\OO(1/n)$ we have 
    $\| f_rB^{r+1}_i-f_rB^{r+1}_j\| \le f_r\cdot 0.7 \Delta' \le 0.7\Delta$. 
    
    Next, both $\|P^k_B\bu-f_rB^{r+1}_i \|$ and $\|P^k_B\bv-f_rB^{r+1}_j \|$ are upper bounded by $o(\Delta)$ from Lemma~\ref{lem:power-SVD-distance}. Combining the two we get that $\| P^k_B\bu-P^k_B\bv\| \le 0.7\Delta+o(\Delta) \le 0.8\Delta$ with probability $1-\OO(1/n)$.
    
    \item Similarly if $v_i \in V_{\ell}$ and $v_j \in V_{\ell'}$ then 
    \[ \|P^k_B\bu-P^k_B\bv \| \ge \| f_rB^{r+1}_i-f_rB^{r+1}_j\|-\|P^k_B\bu-f_rB^{r+1}_i \| - \|P^k_B\bv-f_rB^{r+1}_j \|
    \]
    Here the first term is lower bounded by $f_r1.2\Delta'=1.2\Delta$ and the next two terms are upper bounded by $o(\Delta)$. This implies $\|P^k_B\bu-P^k_B\bv \| \ge 1.1\Delta$ with probability $1-\OO(1/n)$, which completes the proof. 
    
\end{enumerate}

Now, we obtain Lemma~\ref{lem: balanced B} and \ref{lem:power-SVD-distance}, to obtain separation based on $B^r$ and then upper bound $\|P^k_B\bu- f_rB^{r+1}_i\|$.

\subsection{Separation in the Rows of Power Matrix in the \textrm{USSBM} Case}
We start with formalizing the balancedness in the size of the communities. Since each vertex is assigned a community uniformly at random, the expected size of any cluster is $s:=n/k$.
Under the choices of parameters $p,q,k,n$ from Theorem~\ref{thm: centered-svd} we have $n/k \gg \sqrt{n}(\log n)^7$ and then applying Chernoff's bound on the size of the communities we get
\begin{fact}[Size of the clusters]
\label{fact: size}
In the $\USSBM(n,k,p,q)$ model under the condition of Theorem~\ref{thm: main}, for any cluster $V_i$ we have
$\Pr(|V_i-n/k| \ge \frac{1}{(\log n)^5} \cdot n/k) \le n^{-4}$.
\end{fact}

This gives us a relation between the largest and smallest cluster in the $\USSBM$ framework. Let $\smax$ and $\smin$ be the sizes of the largest and the smallest clusters respectively. We then have the following lemma. 

\begin{lemma}
\label{lem: balanced B}
Let $\Delta'=0.5\sqrt{s} (p-q)^{r+1} (s)^{r}$. 
In the $\USSBM$ case under the condition of Theorem~\ref{thm: centered-svd},
if $v_i$ and $v_j$ belong to the same cluster then $\|B^{r+1}_i-B^{r+1}_j\| \le 0.7\Delta'$
and $\|B^{r+1}_i-B^{r+1}_j\| \ge 1.2\Delta'$ otherwise.
\end{lemma}

\begin{proof}

Recall that 
\[
\left| \|B^{r+1}_i-B^{r+1}_j\|-\|L^{r+1}_i-L^{r+1}_j\| \right| 
\le 2( \|M\|_{row}+ \|M'\|_{row} +\|R^{r+1}\|_{row})
\]

Now combining Corollary~\ref{cor: norm-M}, Lemma~\ref{lemma :norm: Rr} and Corollary~\ref{cor: norm: M'} we have that with probability $1-\OO(1/n)$,
$2( \|M\|_{row}+ \|M'\|_{row} +\|R^{r+1}\|_{row}) \le 
0.1 \sqrt{\smax} (p-q)^{r+1} (\smax)^{r}$.
Fact~\ref{fact: size} implies that for any $r \le \log n$ with probability $1-\OO(n^{-4})$, 
\begin{itemize}
    \item $(\smax)^{r+1} \le s^{r+1} \cdot (1+\log^{-5}n)^{r+1} \le s^{r+1} \cdot (1+\log^{-4}n) \le 1.01s^{r+1}$.
    
    \item $(\smin)^{r+1} \ge s^{r+1} \cdot (1-\log^{-5} n)^{r+1} \ge s^{r+1} \cdot (1-\log^{-4} n) \ge  0.99s^{r+1}$.
\end{itemize}

Thus we get with probability $1-\OO(1/n)$,
\[
2( \|M\|_{row}+ \|M'\|_{row} +\|R^r\|_{row}) \le 0.33 \sqrt{s} \cdot (p-q)^{r+1} s^{r} \le 0.7 \Delta'
\]

Next, Lemma~\ref{lemma: norm:L} implies that if $v_i$ and $v_j$ do not belong to the same community then 
$\|L^{r+1}_i-L^{r+1}_j\| \ge \sqrt{\smin} (p-q)^{r+1} (\smin)^{r} \ge 0.99 \sqrt{s} (p-q)^{r+1} s^{r} \ge 1.9\Delta'$ with probability $1-\OO(1/n)$.

Combining this we get that with probability $1-\OO(1/n)$
$\|B^{r+1}_i-B^{r+1}_j\| \le 0.7\Delta'$ if $v_i$ and $v_j$ belong to the same cluster, and $\|B^{r+1}_i-B^{r+1}_j\| \ge 1.9\Delta'-0.7\Delta' \ge 1.2\Delta'$, which completes the proof. 
 
\end{proof}

Next, we connect the SVD projection with the power matrix.

\subsection{Relation Between Eigenspace and Power Matrix}

We aim to show that $\|P^k_B\bu-f_rB^{r+1}_i\|=o(\Delta)$ with high probability.
To obtain this we first show that
$\|f_rB^{r+1}_i\|=\Theta(\Delta)$ and then 
$\|P^k_B\bu-f_rB^{r+1}_i\| =o(\|f_rB^{r+1}_i\|)$.

\begin{lemma}
\label{lem: O-Delta}
Let $\Delta=0.5(p-q)\sqrt{s}$.
Then under the condition of Theorem~\ref{thm: centered-svd} we have 
$\|f_rB_i^{r+1}\|=\Theta(\Delta)$ with probability
$1-\OO(1/n)$ for any $i$.
\end{lemma}
\begin{proof}
Recall that $B^{r+1}_i=L^{r+1}_i+M_i+M'_i+R^{r+1}_i$.
From Section~\ref{sec:prove-power} we have that with probability $1-\OO(1/n)$,
$\|M\|_{row}+\|M'\|_{row}+\|R^r\|_{row} \le 0.15\sqrt{\smax}(p-q)^{r+1}(\smax)^{r} \le 0.16\sqrt{s}(p-q)^{r+1}s^r$. Furthermore, 

\begin{itemize}
    \item $\| L^{r+1}_i \| \ge \sqrt{\smin}(p-q)^{r+1}(\smin)^{r} \ge 0.99\sqrt{s}(p-q)^{r+1}s^r$ with probability $1-\OO(n^{-3})$.

    \item $\| L^{r+1}_i \| \le \sqrt{2\smax}(p-q)^{r+1}(\smax)^{r} \le 2.2\sqrt{s}(p-q)^{r+1}s^r$ with probability $1-\OO(n^{-3})$.
\end{itemize}

Thus, with probability $1-\OO(1/n)$
we have
\begin{align*}
&
\|f_rB^{r+1}_i \| \ge f_r(\|L^{r+1}_i\| -\|M\|_{row}-\|M'\|_{row}-\|R^r\|_{row})
\ge (p-q)^{-r}s^{-r} \cdot 0.5 \sqrt{s}(p-q)^{r+1}s^r
\\
&
\ge 0.5 \sqrt{s}(p-q) \ge 0.5\Delta
\end{align*}

Similarly, with probability $1-\OO(1/n)$,
$\|f_rB^{r+1}_i \| \le 5\Delta$. This completes the proof.

\end{proof}

\noindent Furthermore, let us also upper bound $\Delta$. 
Recall that $(p-q)s \gg \sqrt{p(1-q)} \sqrt{n} (\log n)^7$. This implies
\begin{equation}
\label{eq: Delta-bound}
\Delta=0.5(p-q)\sqrt{s} \ge (\log n)^7/\sqrt{n}
\end{equation}

We then have the following bound on $\|f_rB^{r+1}_i-P^k_B\bu\|$.

\begin{lemma}
\label{lem:power-SVD-distance}
Let $\Delta=0.5(p-q)\sqrt{s}$, $f_r=\frac{1}{(p-q)^rs^r}$ and the relation between $n$ and $k$ be as stated in Theorem~\ref{thm: centered-svd}. 
Let $\bu$ be the $i$-th column of $B$.
Then with probability $1-\OO(n^{-4})$ we have
\[
\|P^k_B\bu-f_rB^{r+1}_i\| =o(\Delta)
\]
\end{lemma}

\begin{proof}
We first note that $B^{r+1}_i=B^r\bu$. This follows from the fact that $\bu$ is the $i$-th column of $B$ and that $B$ is a symmetric matrix. Then we observe 
$P^k_B\bu-f_rB^r\bu$. 

Let $\bm{p_i}, 1 \le i \le n$ be the $n$ the orthonormal eigenvectors of $B$ in descending order of eigenvalues $\lambda_1(B) \ge \ldots \ge \lambda_n(B)$. 
We denote $\lambda_i(B)$ simply as $\lambda_i$. Then any column $\bu$ can be represented as 
$\bu=\sum_{j=1}^n \langle \bm{p_j},\bu \rangle \bm{p_j}$.
It implies $B^r\bu= \sum_{j=1}^n (\lambda_i)^r \langle \bm{p_j},\bu \rangle \bm{p_j}$.

On the other hand, as $P^k_B$ is the projection operator onto the top $k$ eigenvectors of $B$, we have
$P^k_B = \sum_{i=1}^k \langle \bm{p_{i}},\bu \rangle \bm{p_i}$.
Substituting these descriptions we get
\begin{align}
\label{eq: Br-PkB}
f_rB^r\bu -P^k_B\bu
&=
\sum_{j=1}^n f_r(\lambda_j)^r \langle \bm{p_j},\bu \rangle \bm{p_j}- \sum_{i=1}^k \nonumber \langle \bm{p_{i}},\bu \rangle \bm{p_i}
\\
&
=
\sum_{i=1}^k (f_r(\lambda_i)^r -1) \langle \bm{p_{i}},\bu \rangle \bm{p_i}
+\sum_{j=k+1}^n f_r(\lambda_j)^r \cdot \langle \bm{p_j},\bu \rangle \bm{p_j}
\nonumber \\
&
\le
\sum_{i=1}^k \left|f_r(\lambda_i)^r -1\right| \langle \bm{p_{i}},\bu \rangle \bm{p_i}
+\sum_{j=k+1}^n f_r(\lambda_j)^r \cdot \langle \bm{p_j},\bu \rangle \bm{p_j}
\nonumber \\
\implies
\|f_rB^r\bu -P^k_B\bu\|^2
&
\le 
\sum_{i=1}^k (\left|f_r(\lambda_i)^r -1\right|^2 \langle \bm{p_{i}},\bu \rangle)^2
+
\sum_{j=k+1}^n
(f_r(\lambda_j)^{2r} \langle \bm{p_j},\bu \rangle)^2
\end{align}
 
Now we analyze $f_r(\lambda_i)^r$.
Recall that $\lambda_i=\lambda_i(B)$ and $B=L+R$. Then Weyl's inequality implies 
$ \lambda_i(L)-\|R\| \le \lambda_i(B) \le \lambda_i(L)+\|R\|$ for $i \le k$ and $\lambda_i(B) \le \|R\|$ for $i>k$ as $L$ has rank $k$.
In this direction, $\lambda_i(L)=(p-q)|V_i|$ where $V_i$ is the $i$-th largest community. Furthermore, 
$\|R\| \le C_0\sqrt{p(1-q)}\sqrt{n}$ with probability $1-\OO(n^{-3})$. We also have
$(p-q)\smin \ge C_0 \sqrt{p(1-q)} \sqrt{n} (\log n)^7$.
We then have two observations.

\begin{enumerate}
    \item Let $i \le k$. Then $| \lambda_i(B) -(p-q)s | \le \|R\| + (p-q)(\smax-s)+(p-q)(s-\smin)$
    Then with probability $1-\OO(n^{-3})$ this implies 
    $| \lambda_i(B) -(p-q)s | \le (p-q)s(\log n)^{-6}+ 2(p-q)s(\log n)^{-4} \le 
    (p-q)s(\log n)^{-3}$. This implies $f_r(\lambda_i)^r = 1 \pm (\log n)^{-3}$
    and thus $(f_r(\lambda_i)^r-1)^2=o((f_r\lambda_i)^{2r})$.
    
    \item Let $i>k$. Then $\lambda_i \le (p-q)s(\log n)^{-5}$ with probability $1-\OO(n^{-3})$. Then with the same probability $f_r(\lambda_i)^r \le
    f_r(p-q)^rs^r (\log n)^{-5r} \le (\log n)^{-5r} \ll n^{-2}$.
\end{enumerate}

Substituting this in Equation~\eqref{eq: Br-PkB} we get

\begin{align*}
\|f_rB^r\bu -P^k_B\bu\|^2
= &
\sum_{i=1}^k o(f_r(\lambda_i)^r )^2 (\langle \bm{p_{i}},\bu \rangle)^2
+
\sum_{j=k+1}^n
\frac{1}{n^{4}} (\langle \bm{p_j},\bu \rangle)^2
\\
=&
o(\|f_rB^{r+1|}\|^2) + o(n^{-2}) &\text{[From definition of $f_rB^{r+1}$ and $\|u\|^2 \le n$]}
\\
=&
o(\Delta^2) & \text{[Lemma~\ref{lem: O-Delta} and Equation~\eqref{eq: Delta-bound}]}
\end{align*}

This implies that with probability $1-\OO(n^{-3})$ we have $\|f_rB^{r+1}_i-P^k_B\bu\|=o(\Delta)$ for any $i$
and then applying union bound on $i$ completes the proof. 
\end{proof}

\section{Proving Entry-wise Concentration Bounds of via Random Partition}
\label{sec:proof}

In this section we bound the absolute value of the
entries of $R^tL$ for any $t \le \log n$, proving Lemma~\ref{lem:last-step}. Let us first recall the statement of the lemma. We want to show that with high probability ($1-\OO(n^{-3})$) the absolute value of entries of $R^tL$, denoted as $(R^tL)_{a,b}$ is upper bounded by
\[
\left| (R^tL)_{a,b}  \right| \le C_2 \sqrt{p}(\log n)^6 \cdot (p-q)\sqrt{\smax} \cdot ((p-q)\smax)^{t-1}  
\]
for some constant $C_2$. 
We first expand $(R^tL)_{a,b}$ as follows. 
\[
(R^tL)_{a,b}=
\sum_{\ell_1, \ldots ,\ell_t} R_{a,\ell_1}R_{\ell_1,\ell_2}\cdots R_{\ell_{t-1},\ell_t}L_{\ell_t,b}
\]

Thus, each entry of $R^tL$ is a sum of multivariate monomials, with the degree of each monomial being $t+1$. We can denote each monomial with their index list $\Ell:=(\ell_1, \ldots, \ell_t) \in [n]^t$ which establishes
a one-to-one correspondence between monomials $P_{\Ell}$ and index lists $\Ell$. Then we have $(R^tL)_{a,b}=\sum_{\Ell \in [n]^t} P_{\Ell}$.

Against this setting, we break the set of $\Ell$ into several groups based on the similarity of the indices. 
We define an encoding of the lists to a $t$-length array $X$. Then all $\Ell$ that map into the same encoding $X$ form one such group. The encoding is as follows.

\begin{definition}[Encoding index lists to $X$]
\label{def: encoding}
Given any index list $\Ell=(\ell_1, \ldots, \ell_t) \in [n]^t$, we encode it into a $t$-length array $X$ as follows.
\begin{itemize}
    \item Set $X[1]=1$.
    \item For all $i\geq 2$, let $y:= \max \{X[1], \ldots , X[i-1]\}$. If $\ell_{i}\neq \ell_{j}$ for all $j<i$, we set $X[i]=y+1$. Otherwise, $X[i]=X[j]$ for any $j<i$ s.t $\ell_{i}=\ell_j$.  
\end{itemize}
\end{definition}

Let the set of all such $X$ be $\mX$. 
Then the size of $\mX$ follows simply from its definition.

\begin{lemma}
\label{lem: size-of-T}
The size of $\mX$ is upper bounded by $t^t$, i.e. $|\mX|\leq t^t$.
\end{lemma}

Now, corresponding to any $X \in \mX$, we define by $\Z(X)$ the sum of monomials $P_{\Ell}$ such that $\Ell$ maps to that particular $X$. 
Let $t':=\max \{X[1], \ldots, X[t]\}$. Then $\Z(X)$ can be expressed as follows,
\[
\Z(X)=
\sum_{\substack{ \ell_1, \ldots , \ell_{t'} \in ([n])^{t'} \\ \text{ s.t } \ell_{i_1} \ne \ell_{i_2} \forall i_1, i_2 \in [t'] }} R_{a,\ell_{X[1]}}R_{\ell_{X[1]},\ell_{X[2]}} \cdots R_{\ell_{X[t-1]},\ell_{X[t]}}L_{\ell_{X[t]},b}
\]
And we have the inequality 
\begin{equation}
\label{eq: new-break}
|(R^L)_{a,b}| \le \sum_{X \in \mX} \abs{\Z(X)} \end{equation}

We prove upper bounds on $\abs{\Z(X)}$ for any arbitrary $X$ with probability
$1-(t^{-t} \cdot n^{-5})$, and then use Lemma~\ref{lem: size-of-T} to upper bound the total value using a union bound.

We first prove it for the simplest case, which is for the encoding $X$ where all indices are different. That is, the corresponding $X$ has $X[i]=i, 1 \le i \le t$. We denote $W^t_{a,b}:=\Z(X)$ for this particular $X$, and bound $\abs{W^t_{a,b}}$. Then we extend our result for a generic $X \in \mX$.

\subsection*{The case with all different indices.}

In this part we denote $\ell_0:=a$ and $\ell_{t+1}:=b$ for ease of presentation. We want to to upper bound the sum 
\[
\left|\W \right|= \left| \sum_{(\ell_1,  \ldots , \ell_{t}): \ell_x \ne \ell_y \forall x,y}
R_{\ell_0,\ell_1} \cdots R_{\ell_{t-1},\ell_{t}}L_{\ell_t, \ell_{t+1}}
\right|
\]

We use a random partition idea to bound this summation. 
Let $[n]:=\{1, \ldots , n\}$.
We partition $[n]$ into $t$ many sets 
$T=\{T_1, \ldots ,T_t\}$ by putting each element into one of the sets uniformly at random.
Then we  consider the summation
\[\W(T)=
\sum\limits_{\ell_1\in T_1, \ldots ,\ell_t \in T_t} R_{\ell_0,\ell_1}R_{\ell_1,\ell_2}\cdots R_{\ell_{t-1},\ell_t}L_{\ell_t,\ell_{t+1}}\]

As $T$ is a  partition formed uniformly at random, the probability that any monomial of $W^t_{\ell_0,\ell_{t+1}}$ is present in $W^t_{\ell_0,\ell_{t+1}}(T)$ is $(1/t)^t$. 
Then we have 
\begin{align*}
&\W =t^t \cdot \underset{T}{\mathbb{E}} [ \W(T)]
\\
\implies 
&
\abs{\W}=t^t\abs{\underset{T}{\mathbb{E}} [ \W(T)]}
\end{align*}
We then have the following Lemma. 
\begin{lemma}
\label{lem: W-total}
For any $0 <t < \log n $ and $\ell_0,\ell_{t+1}$
we have 
$\left| \W \right| \le t^t \cdot  \underset{T}{\mathbb{E}} \left[\left| \W(T) \right| \right]$. 
\end{lemma}

Against this setup we first upper bound $\left| \W(T) \right|$ for a fixed $T$, and then 
obtain $\underset{T}{\mathbb{E}} [|\W(T)|]$
using the higher moment method.

\paragraph{Analyzing $\abs{\W(T)}$ for a fixed $T$:}
Since all of the random variables in any monomial is different, we can write
\[
\left| \W(T) \right|= \left| \sum_{\ell_1 \in T_1, \ldots, \ell_{t-1} \in T_{t-1}} R_{\ell_0,\ell_1} \cdots R_{\ell_{t-2},\ell_{t-1}}
\sum_{\ell_t \in T_t} R_{\ell_{t-1},\ell_t}L_{\ell_t,\ell_{t+1}}
\right| \]
To bound this term we define
\[
F(i,\ell_{i}) = \sum_{\ell_{i+1}\in T_{i+1},\dots,\ell_{t}\in T_t} R_{\ell_{i},\ell_{i+1}}\cdots R_{\ell_{t-1},\ell_{t}}L_{\ell_{t},\ell_{t+1}}
\]
This implies for every $i$ and $\ell_{i-1}$,
$
\abs{F(i-1,\ell_{i-1})} = \abs{\sum_{\ell_{i}\in T_{i}} R_{\ell_{i-1},\ell_{i}}F(i,\ell_{i})}.$ 
We plan to recursively upper bound $\abs{F(i,\ell_{i})}$ in the descending order of $i$, starting with $F(t-1,\ell_{t-1}):=\sum_{\ell_t \in T_t} R_{\ell_{t-1},\ell_t}L_{\ell_t,\ell_{t+1}}$ and ending at $\abs{F(0,\ell_0)}$. We upper bound this using the fact that since 
$R_{\ell_{i-1},\ell_{i}}$ is a zero mean random variable, so is $R_{\ell_{i-1},\ell_{i}}F(i,\ell_{i})$. Then 
an upper bound on $\abs{F(i,\ell_{i})}$ allows us to apply Chernoff bound on this sum. 

Recall that any column of $L$ has at most $\smax$ non zero ( each $(p-q)$ ) entry and $R_{i,j}$ is a zero-mean independent random variable which is either $1-p$ with probability $p$ and $-p$ otherwise, or $1-q$ with probability $q$, or $-q$ otherwise.
Then Chernoff bound (Theorem~\ref{thm: chernoff}) implies the following.

\begin{corollary}
\label{cor:chernoff}
For all $\ell_{t-1}\in T_{t-1}$,
\[
\Pr(\abs{F(t-1,\ell_{t-1})} \ge 96(p-q)\sqrt{p} \sqrt{\smax} (\log n)^3) \le n^{-4(\log n)^2}.
\]
\end{corollary}

Then applying the definition recursively we have the following lemma. 

\begin{lemma}
\label{lem:F}
Let $\abs{F(i,\ell_{i})}, 1 \le i <t-1$ be less than $\beta$ for all $\ell_{i} \in T_{i}$ with probability $p_{\beta}$. 
Then 
\[
\Pr \left( \abs{F(i-1,\ell_{i-1})} \le \beta \cdot 96 \sqrt{p} \sqrt{n} (\log n)^3 \right) 
\ge (1-n^{-5(\log n)^2})p_{\beta}
\]
\end{lemma}
\begin{proof}
From the definition of $F(i-1,\ell_{i-1})$ we have \[
\abs{F(i-1,\ell_{i-1})} = \abs{\sum_{\ell_{i} \in T_{i}} R_{\ell_{i-1},\ell_{i}} F(i,\ell_{i})}
\]
Here $R_{\ell_{i-1},\ell_{i}} F(i,\ell_{i})$ are zero mean random variables, and with probability $p_{\beta}$, $|F(i,\ell_{i})|\leq\beta$ for all $\ell_i$. Then from the Chernoff bound we have
\[
\Pr \left( \abs{\sum_{\ell_i \in T_i} \ R_{\ell_{i-1},\ell_{i}} F(i,\ell_{i})}
\ge \beta\cdot 96\sqrt{p}\sqrt{|T_i|}(\log n)^3 
\right)
\le n^{-5 (\log n)^2}
\]
Then applying $|T_i| \le n$ completes the proof.
\end{proof}

Next using the recursive definition of $F(i,\ell_i)$ along with Lemma~\ref{lem:F}
we upper bound $|\W(T)|$.

\begin{lemma}
\label{lem: first fixed T}
For any fixed partition $T$ and $t \le \log n$ the sum $\abs{\W(T)}$ is upper bounded by
$C_t:=(96\sqrt{p} (p-q) \sqrt{\smax} (\log n)^3 ) \cdot (96 \sqrt{p} \sqrt{n} (\log n)^3 )^{t-1}$
with probability $1-n^{-3(\log n)^2}$.
\end{lemma}
\begin{proof}
From Corollary~\ref{cor:chernoff}  we have
that 
\[
\Pr \left(\abs{F(t-1,\ell_{t-1})} \le 96\sqrt{p} (p-q) \sqrt{\smax} (\log n)^3 \right) \ge 1-n^{-4(\log n)^2} 
\]
for all values of $t-1$. 
Then application of Lemma~\ref{lem:F} $t-1$ 
times give us
\[
\Pr \left( \abs{F(0,\ell_0)} \le (96\sqrt{p} (p-q) \sqrt{\smax} (\log n)^3 ) \cdot (96 \sqrt{p} \sqrt{n} (\log n)^3 )^{t-1} \right)  \ge 1-t^2n^{-4 (\log n)^2}    
\]
and putting $t \le \log n$
completes the proof. 
\end{proof}

Furthermore $\W(T)$ is the sum of $n^t$ many monomials for any fixed $T$, and the absolute value of each monomial is less than $p-q$. This gives us the following fact.

\begin{fact}
\label{fact-W:max}
For any fixed partition $T$ of $[n]$ 
into $t$ many sets, $\abs{\W(T)} \le (p-q)n^t$.
\end{fact}

Now we analyze  $\underset{T}{\mathbb{E}} [|\W(T)|]$ 
using the higher moment method.

\paragraph{Upper bounding the expectation.}
Let $\beta$ be the number of possible partitions
of $[n]$ into $T$ many uniformly formed sets.
Since each partition is chosen with equal probability, we have
$S:=\underset{T}{\mathbb{E}} [|\W(T)|]= \frac{1}{\beta} \sum_T |\W(T)|$.
Then Markov's inequality implies $\Pr(S \ge \gamma ) \le \frac{\mathbb{E}[S^c]}{\gamma^c}$ for any $c>0$. We are now left with needing an upper bound on $\mathbb{E}[S^c]$. 

\begin{lemma}
\label{lem: expect-S-C}
For any $0< c < \log n$, we have $\mathbb{E}[S^c] \le  (C_t)^c $
\end{lemma}
\begin{proof}
We have 
\begin{align*}
\mathbb{E}[S^c] = \mathbb{E} \left[ \left( 1/\beta \sum_{T} |\W(T)| \right)^c \right] = (1/\beta)^c \sum_{T_1, \ldots ,T_c}
\mathbb{E} \left[ |\W(T_1)| \cdot |\W(T_2)| \cdots |\W(T_c)| \right]
\end{align*}
  
We now analyze
$\mathbb{E} \left[ |\W(T_1)| \cdot |\W(T_2)| \cdots |\W(T_c)| \right]$.
From Lemma~\ref{lem: first fixed T} we have 
that for any fixed $T$,
$\Pr(|\W(T)| \ge C_t ) \le n^{-3 (\log n)^2}$.
This implies 
\[
\Pr \left( \prod_{i \in [c]}\abs{\W(T_i)} \ge C_t^c \right) \le c \cdot n^{-3 (\log n)^2} \le n^{-2 (\log n)^2}
\]
From Fact~\ref{fact-W:max} we have $|\W(T)| \le (p-q)n^t$. 
Then using $ct \le (\log n)^2$ 
we get that for  $T_1, \ldots ,T_c$,
\[
\mathbb{E} \left[ |\W(T_1)| \cdot |\W(T_2)| \cdots |\W(T_c)| \right] \le (1-n^{-2 (\log n)^2})(C_t)^c 
+ n^{-2 (\log n)^2} (p-q)^cn^{ct}
\le (C_t)^c, 
\]
as $c,t \le \log n$. Since there are a total of $\beta^c$ terms, the lemma then follows.
\end{proof}

Then combining with higher moment method we have the following result by setting $\gamma=32 \cdot \log n \cdot C_t$ and $c=\log n$.

\begin{corollary}
\label{cor:expect-S}
For any $\ell_0, \ell_{t+1}$ and $t<\log n$ we have $|\W| \le t^t \cdot 32 \cdot \log n \cdot C_t$
with probability $1-\log n^{\log n}n^{-5}$. 
\end{corollary}

This completes analysis of the summation of all monomials where $\ell_1,\ldots ,\ell_t$ are all different.
Next we extend our proof to a generic $X$.

\subsection*{Monomials with arbitrary individual degrees}

Next we bound the absolute value of the sum of the monomials with any index ordering $X \in \mX$.
Recall the definition
\[
\Z(X)=
\sum_{\substack{ \ell_1, \ldots , \ell_{t'} \in ([n])^{t'} \\ \text{ s.t } \ell_{i_1} \ne \ell_{i_2} \forall i_1, i_2 \in [t'] }} R_{a,\ell_{X[1]}}R_{\ell_{X[1]},\ell_{X[2]}} \cdots R_{\ell_{X[t-1]},\ell_{X[t]}}L_{\ell_{X[t]},b}
\]
and that $t':=\max \{ X[1], \ldots , X[t]\}$.
Here we define the representative monomial 
\[
\mP{t'}(1,\ldots,t', \ell_1, \ldots, \ell_{t'})
:= R_{a,\ell_{X[1]}}R_{\ell_{X[1]},\ell_{X[2]}} \cdots R_{\ell_{X[t-1]},\ell_{X[t]}}L_{\ell_{X[t]},b}
\]
Let $\Ell_S:=\{\ell_i\}_{i \in S}$ for any $S \subseteq [t']$. Then we have
\[
\Z(X)=
\sum_{\substack{ \ell_1, \ldots , \ell_{t'} \in ([n])^{t'} \\ \text{ s.t } \ell_{i_1} \ne \ell_{i_2} \forall i_1, i_2 \in [t'] }}
\mP{t'} \left( [t'],\Ell_{[t']} \right)
\]
Furthermore, we fix all the indices $\ell_i$ that are to be assigned $a$. Since at most one of the 
index can be assigned as $a$, we define by $\Z(X,i)$ as the sum of monomials of $\Z(X)$ for which $\ell_i=a$. 
Furthermore, $Z(X,0)$ is the case where none of the indices are $a$. Then 
\begin{equation}
\label{eq: z-break-up}
\abs{(R^tL)_{a,b}} \le \sum_{X \in \mX} \sum_{0 \le i \le t'} \abs{\Z(X,i)} 
\end{equation}

Here we upper bound $\abs{\Z(X,0)}$ for any arbitrary $X$ using our analysis of $W^t_{a,b}$ and then extend it to complete the proof.
First we describe a sketch of our proof.
\paragraph{Sketch of proof}

We start with a random partition of $[n] \setminus \{a\}$ into $t'$ many sets $T=\{T_1, \ldots , T_{t'}\}$ and upper bound the term
\[
\Z(X,0,T):=\sum_{\substack{\ell_i \in T_i \\ 1 \le i \le t'}} 
\mP{t'} \left( [t'],\Ell_{[t']} \right)
\]
Then, as before we have 
\[
\Z(X,0)= (t')^{t'} \cdot \underset{T}{\mathbb{E}} \left[ \Z(X,0,T) \right]
\implies
\abs{Z(X,0)} \le t^t \cdot \underset{T}{\mathbb{E}} \left[\abs{\Z(X,0,T)} \right]
\]

We upper bound $|\Z(X,0,T)|$ with high probability and then the upper bound on expectation follows from the higher moment method. Let us now focus on
$\Z(X,0,T)$.

In case of $W^t_{a,b}$, we start with a sum of $t+1$ degree monomials, and recursively bound a certain sum based on the index $\ell_i$ at each round, starting from $\ell_t$ and proceeding in the decreasing order. After each round, we are left with bounding a sum of monomials of degree one less than the previous round. 

In contrast, for $\Z(X,0,T)$ we start with $\ell_{t'}$ and go down. Since some indices may appear in multiple random variables, the degree of the monomial may go down arbitrarily after each step. The crux of the idea is to ensure that after we have taken a sum on $\ell_i$, there are still random variables with $\ell_j, j<i$ as their indices for all $j$. This is ensured by the way our encoding scheme is defined. Now we lay out the details. 

\paragraph{Upper bounding $\abs{\Z(X,0,T)}$.}

Note that in $\Z(X,0,T)$, the domain of each index $\ell_i$ is unique (as $T$ is a disjoint partition of $[n] \setminus \{a\}$). 
We also denote the monomial $\mmP(S,\Ell_S)$ simply as $\mmP_S$ for ease of representation.
Then we have the following definition. 
\begin{definition}
\label{def: P-break}
Let $S_t=[t']$. Then, given the monomial $\mP{t'} \left( {S_{t'},\Ell_{S_{t'}}} \right)$,
\begin{enumerate}
    \item $\mP{i,i}$ denotes the product of variables in $\mP{i}$ that have $\ell_i$ as one of their indices.
    
    \item $\mP{i-1}$ denotes the monomial obtained by removing $\mP{i,i}$ from $\mP{i}$.
\end{enumerate}

That is,
\[
\mP{i}_{S_i}= \mP{i-1}_{S_{i-1}} \cdot \mP{i,i}_{S_i'}
\]

where $S_i$ is the set of indices present in the variables of the monomial $\mP{i}$ and $S_i'$ is the set of indices present in the variables of the monomial $\mP{i,i}$.
\end{definition}

Here note that $\ell_i \notin S_{i-1}$ and if $\deg (\mP{i,i}) >0$, then $\ell_i \in S_i'$.
This in turn implies the following identity.
\begin{equation}
\label{eq: P-iterate}
\mP{t'}_{S_t}= \prod_{1 \le i \le t'}
\mP{i,i}_{S_i'}
\end{equation}

Then we have a simple lemma which is a consequence of our encoding scheme~\ref{def: encoding}.

\begin{lemma}
\label{lem: recurse-degree}
For any $1 \le i \le t'$, 
$\deg \left( \mP{i,i}_{S_i'} \right) \ge 1$.
\end{lemma}
\begin{proof}
This can be proven by induction. 
We start with $\mP{t'}$. Recall that $\mP{t'-1}$ is obtained by removing all variables $R_{\ell_{X[i],X[i+1]}}$ or $L_{X[t],b}$ from $\mP{t'}$ that contain $\ell_{t'}$ as an index. 
Furthermore, in the monomial $\mP{t'}$, $\ell_{t'}$ has the right-most first appearance, as $t'$ has the highest value in the corresponding $X$. Thus, the first variable from left in which $\ell_i, i<t'$ is first present is strictly left of the first appearance of $\ell_{t'}$. Thus, all such variables are present in $\mP{t'-1}$.

Similarly consider any $\mP{i}$ such that it has at least one variable with $\ell_j$ as an index for any $j \le i$. Then if $\mP{i,i}$ is removed, the first variable in which $\ell_j, j<i$ appears is still present in $\mP{i-1}$. 
This completes the proof.

\end{proof}
Furthermore, it is also easy to see that if $X[t]=\alpha$ then $\deg \left( \mP{\alpha,\alpha} \right) \ge 2$, as it is product of at least $L_{\ell_{\alpha},b}$ and some $R_{\ell_{\alpha},\ell_{i'}}$. 
Having established this setup, we now upper bound $\abs{\Z(X,0,T)}$. To this end we define 
\begin{equation}
\label{eq: F-Z}
F(S_i,\Ell_{S_i}):=
\sum_{\ell_{i+1} \in T_{i+1}, \ldots, \ell_{t'} \in T_{t'} } 
\left(\prod_{t' \ge j > i} 
\mP{j,j}(S_j',\Ell_{S_j'})
\right)
\end{equation}
Then, 
\[
F(S_{i-1},\Ell_{S_{i-1}})=
\sum_{\ell_{i} \in T_{i}} \mP{i,i}(S_{i}',\Ell_{S_{i}'})F(S_{i},\Ell_{S_{i}})
\]
and $\Z(X,0,T)=F(S_{0},\Ell_{S_{0}})$ where $S_{0}=\emptyset$. 

We first obtain a bound on sum of variables of the form $\sum_{i \in T_i}R^c_{i,j}$ where $c \ge 2$. We shall use this result as in this part we deal with sum of monomials where the variables may have higher individual degrees. 
\begin{lemma}
\label{lem: chernoff-higher degree}
Let $c \ge 2$ and $T \subseteq [n]$ so 
that $p(1-q) \cdot |T| \ge 1$. 
Then with probability $1-n^{-4 \log(n)^2}$ we have
\[
\abs{\sum_{i \in T} R^c_{i,j}} \le 192 p(1-q) \cdot |T| \cdot (\log n)^3
\]
\end{lemma}
\begin{proof}
Recall that $R_{i,j}$ is a zero mean random variable with one of the two distributions. 

\begin{enumerate}
    \item If $i \sim j$, then $R^c_{i,j}$ is $(1-p)^c$ with probability $p$, and $(-p)^c$ with probability $1-p$. Then 
    \[
    E[R^c_{i,j}] = (1-p)^cp+(-p)^c(1-p)=
    p(1-p) \cdot ((1-p)^{c-1}-(-p)^{c-1})
    \]
    Since $\abs{(1-p)^{c-1}-(-p)^{c-1}}\le 1$, 
    we have $\abs{E[R^c_{i,j}]} \le p(1-p)$.
    
    \item If $i \nsim j$ then $R^c_{i,j}$ is 
    $(1-q)^c$ with probability $q$ and $(-q)^c$ otherwise. Then $\abs{ E[R^c_{i,j}]} \le q(1-q)$.
\end{enumerate}

Thus, the random variable is two valued with the two points being less than $1$ distance apart, and has expectation upper bounded as $|E[R^c_{i,j}]| \le p(1-q)$. Then Chernoff's theorem (Theorem~\ref{thm: chernoff}) implies that with probability $1-n^{-4 (\log n)^2}$,
\begin{align*}
&-p(1-q)|T| - 192\sqrt{p}\sqrt{|T|} (\log n)^3 \le 
\sum_{i \in T} R^c_{i,j} \le p(1-q)|T| + 192\sqrt{p}\sqrt{|T|} (\log n)^3
\\
\implies 
&
\abs{\sum_{i \in T} R^c_{i,j}} \le p(1-q)|T| + 192\sqrt{p}\sqrt{|T|} (\log n)^3
\\
\implies 
&
\abs{\sum_{i \in T} R^c_{i,j}} \le 
192p(1-q)|T| (\log n)^3
\end{align*}

\end{proof}

We then have the following result

\begin{lemma}
\label{lem: Z-iteration}
Let $\deg( \mP{i,i})=w_i$.
If $\abs{F(S_i,\Ell_{S_i})} \le \beta$ with probability $p_{\beta}$, then
\begin{enumerate}
    \item If $X[t]=i$, then with probability $p_{\beta}(1-n^{-3 (\log n)^2})$
    \[
    \abs{F(S_{i-1},\Ell_{S_{i-1}})} \le \beta \cdot 
    96(p-q)\sqrt{p}\sqrt{\smax}(\log n)^3 \cdot (96\sqrt{p}\sqrt{n}(\log n)^3)^{w_i-1}
    \]
    
    \item Otherwise with probability $p_{\beta} (1-n^{-3 (\log n)^2})$
    \[
    \abs{F(S_{i-1},\Ell_{S_{i-1}})} \le \beta \cdot    (96\sqrt{p}\sqrt{n}(\log n)^3)^{w_i}
    \]

\end{enumerate}

\end{lemma}

\begin{proof}

Equation~\eqref{eq: F-Z} implies that 
\[
F(S_{i-1},\Ell_{S_{i-1}})=
\sum_{\ell_{i} \in T_{i}} \mP{i,i}(S_{i}',\Ell_{S_{i}'})F(S_{i},\Ell_{S_{i}})
\]

We analyze the first case here, and the second case then follows. As in Lemma~\ref{lem:F}, we $\mP{i,i}(S_{i}',\Ell_{S_{i}'})F(S_{i},\Ell_{S_{i}})$ is a zero mean random variable. And since $X[t]=i$, then for some $S \subseteq [i]$ and $c_j$ being positive integers,
\[
\mP{i,i}= \prod_{j \in S} R^{c_j}_{\ell_i,\ell_j} \cdot L_{\ell_i,b}
\]

Here note that $w_i \ge 2$ as $\mP{i,i}$ contains $L_{\ell_i,b}$ and at least one variable from $R$.
The following events then happen with probability $(1-n^{-3 (\log n)^2})p_{\beta}$.

\begin{itemize}
\item 
If $w_i=2$, then $\mP{i,i}=R_{\ell_i,\ell_{i'}}L_{\ell_i,b}$ for some $i'$ and 
$\abs{F(S_{i-1},\Ell_{S_{i-1}})}$ can be directly upper bounded by
$\beta \cdot 96(p-q)\sqrt{p}\sqrt{\smax}(\log n)^3$ from Lemma~\ref{lem:F} (scaling the variables with $1/\beta$ and applying Chernoff bound).

\item 
Otherwise, if $w_i>2$, $\abs{\mathbb{E}[\mP{i,i}]} \le p(1-q)$. Then Lemma~\ref{lem: chernoff-higher degree} implies that with probability $p_{\beta} \cdot (1-n^{-3 (\log n)^2})$,
\begin{align*}
\abs{F(S_{i-1},\Ell_{S_{i-1}})}&
\le \beta \cdot 96 (p-q) \cdot p(1-q) \smax (\log n)^3 
\\ &
\le 96(p-q)\sqrt{p}\sqrt{\smax}(\log n)^3 \cdot (96\sqrt{p}\sqrt{n}(\log n)^3)^{w_i-1} &\text{[As $\smax \le \sqrt{\smax}\sqrt{n}$]}
\end{align*}
\end{itemize}

The case for $X[t] \ne i$ also follows in the same manner.
\end{proof}

This result gives us an upper bound on $\abs{\Z(X,0,T)}$ by iterating $t'-1$ times through Lemma~\ref{lem: Z-iteration}.
\begin{corollary}
\label{lem-Xbar-Z-first-bound}
Let $t' \le \log n$ and $T=\{T_1, \ldots ,T_{t'} \}$ be any fixed partition. Recall that $C_t=(96\sqrt{p} (p-q) \sqrt{\smax} (\log n)^3 ) \cdot (96 \sqrt{p} \sqrt{n} (\log n)^3 )^{t-1}$. Then we have
\[
\Pr \left(\abs{\Z(X,0,T)} \le C_t \right) \ge 1-n^{-2 (\log n)^2}
\]
\end{corollary}

Finally, to bound the expectation, as before we have
\begin{equation}
\label{eq:Z(X,0,T)}
|\Z(X,0)| \le (t')^{t'} 
\underset{T}{\mathbb{E}} [|\Z(X,0,T)|]    
\end{equation}

Once we have an upper bound for an arbitrary $T$, the same analysis as for $\W$ follows for obtaining $\mathbb{E}_T[|\Z(X,0,T)|]$, where in we get $\mathbb{E}_T[|\Z(X,0,T)|] \le 32 \log n \cdot C_t$ and then get the following bound on $Z(X,0)$ from Equation~\ref{eq:Z(X,0,T)}
(we skip redoing the calculation for brevity). 
\begin{corollary}
\label{cor: Z(X,0)}
\[
\Pr \left( \abs{Z(X,0)} \le t^{t} \cdot 32 \log n \cdot C_t \right) \ge 1- (\log n)^{\log n}n^{-5}
\]
\end{corollary}

The same bound follows for any $\Z(X,i)$ as we merely fix one of the indices to a single value ($a$) and then randomly partition $[n] \setminus \{a\}$ into $T'$ many sets. 
That is, Corollary~\ref{cor: Z(X,0)} implies 
\begin{equation}
\label{eq: Z(X,i)}
\Pr \left( \abs{Z(X,i)} \le t^t \cdot 32 \log n\cdot C_t  \right) \ge 1- (\log n)^{\log n}n^{-5} 
\end{equation}
We can now summarize to complete the proof.

\paragraph{Proof of Lemma~\ref{lem:last-step}}
Combining Equations~\eqref{eq: z-break-up} and \eqref{eq: Z(X,i)} via an union bound on all $(t+1)\cdot t^t$ possibilities, we get with 
probability $1-\OO(n^{-4})$,
\begin{align*}
&|(R^tL)_{a,b} 
\le 32 \cdot (t+1) \cdot t^{2t} \cdot  \log n \cdot C_t
\\ &
\le 
32 \cdot 2^t \cdot (t^2)^t \cdot \log n \cdot
(96\sqrt{p} (p-q) \sqrt{\smax} (\log n)^3 ) \cdot (96 \sqrt{p} \sqrt{n} (\log n)^3 )^{t-1}
\\ &
\le 
32 \cdot 2^{t} \cdot ((\log n)^2)^{t} \log n \cdot
(96\sqrt{p} (p-q) \sqrt{\smax} (\log n)^3 ) \cdot (96 \sqrt{p} \sqrt{n} (\log n)^3 )^{t-1}
&\text{[Substituting $t \le \log n$]}
\\
&
\le 
128 \cdot 96\sqrt{p} (p-q) \sqrt{\smax} (\log n)^6 
\cdot (192 \sqrt{p} \sqrt{n} (\log n)^5 )^{t-1}
\\
&
\le
12228 (p-q)(\log n)^6
\cdot  ((p-q) \sqrt{\smax} ) 
\cdot ((p-q)\smax)^{t-1} 
&\text{[ $(p-q)\smax \gg 192 \sqrt{p} \sqrt{n} (\log n)^5$]}
\end{align*}

This completes the proof.


\begin{thebibliography}{AFWZ20}

\bibitem[Abb17]{abbe2017community}
Emmanuel Abbe.
\newblock Community detection and stochastic block models: recent developments.
\newblock {\em The Journal of Machine Learning Research}, 18(1):6446--6531,
  2017.

\bibitem[ACX13]{ailon2013breaking}
Nir Ailon, Yudong Chen, and Huan Xu.
\newblock Breaking the small cluster barrier of graph clustering.
\newblock In {\em International conference on machine learning}, pages
  995--1003. PMLR, 2013.

\bibitem[AFWZ20]{abbe2020entrywise}
Emmanuel Abbe, Jianqing Fan, Kaizheng Wang, and Yiqiao Zhong.
\newblock Entrywise eigenvector analysis of random matrices with low expected
  rank.
\newblock {\em Annals of statistics}, 48(3):1452, 2020.

\bibitem[AKS98]{alon1998finding}
Noga Alon, Michael Krivelevich, and Benny Sudakov.
\newblock Finding a large hidden clique in a random graph.
\newblock {\em Random Structures \& Algorithms}, 13(3-4):457--466, 1998.

\bibitem[BCLS87]{bui1987graph}
Thang~Nguyen Bui, Soma Chaudhuri, Frank~Thomson Leighton, and Michael Sipser.
\newblock Graph bisection algorithms with good average case behavior.
\newblock {\em Combinatorica}, 7(2):171--191, 1987.

\bibitem[Bop87]{boppana1987eigenvalues}
Ravi~B Boppana.
\newblock Eigenvalues and graph bisection: An average-case analysis.
\newblock In {\em 28th Annual Symposium on Foundations of Computer Science
  (sfcs 1987)}, pages 280--285. IEEE, 1987.

\bibitem[DF89]{dyer1989solution}
Martin~E. Dyer and Alan~M. Frieze.
\newblock The solution of some random np-hard problems in polynomial expected
  time.
\newblock {\em Journal of Algorithms}, 10(4):451--489, 1989.

\bibitem[DK69]{Davis-Kahan}
Chandler Davis and W.~M. Kahan.
\newblock {Some new bounds on perturbation of subspaces}.
\newblock {\em Bulletin of the American Mathematical Society}, 75(4):863 --
  868, 1969.

\bibitem[EBW18]{Eldrige-Unperturbed}
Justin Eldridge, Mikhail Belkin, and Yusu Wang.
\newblock Unperturbed: spectral analysis beyond davis-kahan.
\newblock In Firdaus Janoos, Mehryar Mohri, and Karthik Sridharan, editors,
  {\em Proceedings of Algorithmic Learning Theory}, volume~83 of {\em
  Proceedings of Machine Learning Research}, pages 321--358. PMLR, 07--09 Apr
  2018.

\bibitem[FK81]{furedi1981eigenvalues}
Zolt{\'a}n F{\"u}redi and J{\'a}nos Koml{\'o}s.
\newblock The eigenvalues of random symmetric matrices.
\newblock {\em Combinatorica}, 1(3):233--241, 1981.

\bibitem[HLL83]{holland1983stochastic}
Paul~W Holland, Kathryn~Blackmond Laskey, and Samuel Leinhardt.
\newblock Stochastic blockmodels: First steps.
\newblock {\em Social networks}, 5(2):109--137, 1983.

\bibitem[Hoe63]{Chernoff}
Wassily Hoeffding.
\newblock Probability inequalities for sums of bounded random variables.
\newblock {\em Journal of the American Statistical Association},
  58(301):13--30, 1963.

\bibitem[LR15]{JingconsistentSBM}
Jing Lei and Alessandro Rinaldo.
\newblock Consistency of spectral clustering in stochastic block models.
\newblock {\em The Annals of Statistics}, 43(1):215--237, 2015.

\bibitem[McS01]{mcsherry2001spectral}
Frank McSherry.
\newblock Spectral partitioning of random graphs.
\newblock In {\em Proceedings 42nd IEEE Symposium on Foundations of Computer
  Science}, pages 529--537. IEEE, 2001.

\bibitem[MPZ22]{mukherjee2022recovering}
Chandra~Sekhar Mukherjee, Pan Peng, and Jiapeng Zhang.
\newblock Recovering unbalanced communities in the stochastic block model with
  application to clustering with a faulty oracle.
\newblock {\em arXiv preprint arXiv:2202.08522}, 2022.

\bibitem[Vu05]{vu-norm}
V.~H. Vu.
\newblock Spectral norm of random matrices.
\newblock In {\em Proceedings of the Thirty-Seventh Annual ACM Symposium on
  Theory of Computing}, STOC '05, page 423–430, New York, NY, USA, 2005.
  Association for Computing Machinery.

\bibitem[Vu18]{vu2018simple}
Van Vu.
\newblock A simple svd algorithm for finding hidden partitions.
\newblock {\em Combinatorics, Probability and Computing}, 27(1):124--140, 2018.

\bibitem[Wig58]{wigner1958distribution}
Eugene~P Wigner.
\newblock On the distribution of the roots of certain symmetric matrices.
\newblock {\em Annals of Mathematics}, pages 325--327, 1958.

\bibitem[YP14]{Yun2014AccurateCD}
Seyoung Yun and Alexandre Prouti{\`e}re.
\newblock Accurate community detection in the stochastic block model via
  spectral algorithms.
\newblock {\em ArXiv}, abs/1412.7335, 2014.

\end{thebibliography}
\end{document}